%% file: covario_smooth_Cn_submitted_PLMS_v2.tex
\documentclass[a4paper]{amsart}
\usepackage{graphicx}
\usepackage{url}
\usepackage{enumerate}
\usepackage{mathrsfs}

\newtheorem{theorem}{Theorem}[section]

\newtheorem{lemma}[theorem]{Lemma}
\newtheorem{proposition}[theorem]{Proposition}

\theoremstyle{definition}

\newtheorem{remark}[theorem]{Remark}
\newtheorem{example}[theorem]{Example}

\numberwithin{equation}{section}

\newcommand{\inte}{{\operatorname{int}}}
\newcommand{\cl}{{\operatorname{cl}}}

\newcommand{\supp}{{\operatorname{supp}}}

\newcommand{\width}{{{w}}}

\newcommand{\voln}{{\mathop{\lambda_n}}}
\newcommand{\bigO}{{\operatorname{O}}}
\newcommand{\littleO}{{\operatorname{o}}}
\newcommand{\pa}{{\partial}}
\newcommand{\conj}[1]{\overline{#1}}
\newcommand{\cz}{\overline{\zeta}}
\newcommand{\fou}[1]{\widehat{{#1}}}
\newcommand{\Sn}{S^{n-1}}
\newcommand{\Sndue}{S^{n-2}}
\newcommand{\Sdue}{S^{1}}
\newcommand{\gau}{\tau}
\newcommand{\ds}{{r(n)}}

\newcommand{\Wein}{{\mathop{W}}}
\newcommand{\N}{\mathbb{N}}
\newcommand{\R}{\mathbb{R}}
\newcommand{\C}{\mathbb{C}}

\newcommand{\im}{\operatorname{Im}}
\newcommand{\re}{\operatorname{Re}}
\newcommand{\ii}{\operatorname{i}}
\newcommand{\ee}{{\varepsilon}}
\newcommand{\al}{{\alpha}}
\newcommand{\be}{{\beta}}
\newcommand{\de}{{\delta}}
\newcommand{\ga}{{\gamma}}

\newcommand{\la}{{\lambda}}
\newcommand{\te}{{\theta}}
\newcommand{\ze}{{\zeta}}

\newcommand{\cZ}{{\mathcal Z}}
\newcommand{\cK}{{\mathcal K}}
\newcommand{\cH}{{\mathcal H}}
\newcommand{\cT}{{\mathcal T}}
\newcommand{\jcov}[1]{g_{#1}}

\begin{document}
\title[The covariogram  and Fourier-Laplace transform in $\C^n$]{The covariogram
and Fourier-Laplace\\ transform in $\C^n$}
\author{Gabriele Bianchi}
\address{Dipartimento di Matematica e Informatica ``U.~Dini'', Universit\`a di Firenze, 
Viale Morgagni 67/A, Firenze, Italy I-50134}
\email{gabriele.bianchi@unifi.it}

\begin{abstract}
The covariogram $g_{K}$ of a convex body $K$ in $\R^n$ is the function which associates to each
$x\in\R^n$ the volume of the intersection of $K$ with $K+x$. Determining $K$ from the knowledge of $g_K$ is known as the Covariogram Problem. It is equivalent to determining the characteristic function $1_K$ of $K$ from the modulus of its Fourier transform $\fou{1_K}$ in $\R^n$, a particular instance of the Phase Retrieval Problem.

We connect the Covariogram Problem to two aspects of the Fourier transform $\fou{1_K}$ seen as a function in $\C^n$. 
The first connection is with the problem of determining $K$ from the knowledge of the zero set of $\fou{1_K}$ in $\C^n$. 
To attack this problem T.~Kobayashi studied the asymptotic behavior at infinity of this zero set. 
We obtain this asymptotic behavior assuming less regularity on $K$ and we use this result as an essential ingredient for proving that when $K$ is sufficiently smooth and in any dimension $n$, $K$ is determined by $g_K$ in the class of sufficiently smooth bodies.

The second connection is with the irreducibility of the entire function $\fou{1_K}$. This connection also shows a link between the Covariogram Problem and the Pompeiu Problem in integral geometry.
%
\end{abstract}

\maketitle

\section{Introduction}
Let $H$ and $K$ be convex bodies in  $\R^n$, $n \geq 2$, and let $\voln$ stand for the $n$-dimensional Lebesgue measure.
The \emph{cross covariogram} $\jcov{H,K}$ of $H$ and $K$ is the
function defined for $x\in\R^n$ by
\[
\jcov{H,K}(x)=\voln(H\cap (K+x)).
\]
This function coincides with the convolution of  the characteristic function ${1}_H$ of $H$ with the characteristic function ${1}_{-K}$ of the reflection of $K$ in the origin, that is,
\begin{equation}\label{convoluzione}
g_{H,K} ={1}_H\ast {1}_{-K}.
\end{equation}

The function $g_{K,K}$  was introduced by G.~Matheron in his book~\cite[Section~4.3]{Matheron-1975} on random sets, is denoted by $g_K$ and is called \emph{covariogram} of $K$. Observe that $g_K$ is clearly unchanged by  translations or reflections of $K$ (in this paper the term \emph{reflection}  always means reflection in a point).
The data provided by \( g_K \) can be interpreted in several ways within different contexts, using purely geometric, functional-analytic and probabilistic terminology.
As a result, covariograms of convex bodies and other sets appear naturally in various research areas including convex geometry, image analysis, geometric shape and pattern matching, phase retrieval in Fourier analysis, crystallography and geometric probability. See Baake and Grimm~\cite{BaakeGrimm}, Bianchi, Gardner and Kiderlen~\cite{BiGaKi11} and references therein, Matheron~\cite{Matheron-1975} and Schymura \cite{Schymura-2011}.

The following  problem  was posed by G.~Matheron in 1986 (see~\cite{Matheron8601}) and has received much attention in recent years.

\smallskip

\textbf{Covariogram Problem.} \emph{Does the covariogram determine a
convex body, among all convex bodies, up to translations and
reflections?}

\smallskip

The answer to the Covariogram Problem is positive for every planar convex body (see Averkov and Bianchi~\cite{averkov-bianchi-2009}), it is positive for convex polytopes in $\R^3$ (see Bianchi~\cite{Bianchi-2009-polytopes}) but the case of a general convex body in $\R^3$ is still open, and in every dimension $n\geq4$ there are examples of nondetermination (see Bianchi~\cite{Bianchi-2005}), as well as positive results in some subclasses of the class of convex polytopes (see Goodey, Schneider and Weil~\cite{Goodey-Schneider-Weil-1997}).

The \emph{Phase Retrieval Problem} asks for the determination of  a function $f\in L^2(\R^n)$ with compact support from the knowledge of the modulus of its of Fourier transform $\fou{f}(x)$ for $x\in\R^n$, up to the inherent ambiguities. 
In view of \eqref{convoluzione} the Fourier transform of $g_K$ coincides with $|\fou{1_K}|^2$. Therefore the Covariogram Problem coincides with a particular instance of the Phase Retrieval Problem. 

\smallskip

\textbf{Covariogram Problem} (alternative form). 
\emph{Does the modulus $|\fou{1_K}(x)|$, for $x\in\R^n$, determine the convex body $K$, among all convex bodies, up to translations and reflections?}

\smallskip

In this paper we connect the Covariogram Problem to some problems regarding the Fourier transform $\fou{1_K}$ seen as a function in $\C^n$.

The first connection is with a problem related to the zero set  $\cZ(K)=\{\zeta\in\C^n : \fou{1_K}(\ze)=0\}$. 
This set has been studied in the literature, for instance  for the role that it plays in attempts to solve the famous Pompeiu Problem, a long-standing open problem in integral geometry (see, for instance, Berenstein~\cite{Berenstein-1980} and Garofalo and Segala~\cite{Garofalo-Segala-1991}). 
More recently Benguria, Levitin and Parnovski~\cite{Benguria-Levitin-Parnovski-2009} has connected $\cZ(K)$ to properties of some eigenvalues of the Laplacian. 
Here we focus on the studies of T.~Kobayashi (see~\cite{Kob0, Kob1, Kob2}) regarding the geometric information about $K$ contained in $\cZ(K)$. In 1986 Kobayashi \cite{Kob0} has posed the following problem.

\smallskip

\textbf{Problem 2.} 
\emph{Does the zero set $\cZ(K)=\{\zeta\in\C^n : \fou{1_K}(\ze)=0\}$ determine the convex body $K$, among all convex bodies, up to translations?}

\smallskip
(Note that a translation of $K$ leaves $\cZ(K)$ unchanged.) In the class of $C^\infty_+$ convex bodies (the subscript $+$ means that $\pa K$ is assumed to have Gauss curvature positive everywhere) Problem 2 has been solved in the planar case \cite{Kob1}, but is still open for $n\geq3$. 
In connection with Problem~2 Kobayashi~\cite{Kob1, Kob2} studies,  in any dimension and only in the case of $C^\infty_+$ convex bodies, the asymptotic behavior at infinity of $\cZ(K)$. 
It turns out that this asymptotic behaviour contains information about the width function of $K$ and the ratio of the Gauss curvatures of $\pa K$ at antipodal points (see~Theorem~\ref{teo_kobayashi} of this paper for the precise statement).

The Covariogram Problem and Problem~2 have different origins and have not interacted so far. In this paper we bring an idea used for Problem~2 to obtain new results for the Covariogram Problem. 
To this end we first prove Kobayashi's result regarding the asymptotics of $\cZ(K)$ under lower regularity assumptions (we lower these assumptions from $K\in C^\infty_+$ to $K\in C^\ds_+$, where  $\ds$ is as in Theorem~\ref{teo_cov_smooth}). Then we use this extension as a key to prove a positive answer to the Covariogram Problem for $C^\ds_+$ convex bodies in every dimension.
\begin{theorem}\label{teo_cov_smooth}Let $n\geq 2$ and define $\ds=8$ when $n=2,4,6$, $\ds=9$ when $n=3,5,7$
and $\ds=[(n-1)/2]+5$ when $n\geq8$. Let $H$ and $K$  be  convex bodies in $\R^n$ of class $C^\ds_+$. Then $g_H=g_K$ implies that $H$ and $K$ coincide, up to translations and reflections.
\end{theorem}

In the class of $C^2_+$ convex bodies in $\R^n$, the Covariogram Problem has been solved for $n=2$ more that ten years ago~\cite{Bianchi-Segala-Volcic-2002} but it is still open for any  $n\geq3$.
It can be proved that if $K$ is in this class then, for each $u\in \Sn$, $g_K$ provides the nonordered pair $\{\gau_K(u),\gau_K(-u)\}$ consisting in the Gauss curvature of $\pa K$ at the points of $\partial K$ with outer normal $u$ and $-u$. 
Thus if $H$ is in the class $C^2_+$ and $g_H=g_K$, the continuity of the curvature implies that given any component $V$ of $\{u\in \Sn : \gau_K(u)\neq \gau_K(-u)\}$, after possibly a reflection of $H$, we have
\begin{equation}\label{curvatures_locally_equal}
\gau_H(v)=\gau_K(v)\quad\text{for each $v\in V$.}
\end{equation}
If~\eqref{curvatures_locally_equal} were true for each $v\in\Sn$ then $H$ and $K$ would coincide, up to a translation, by the uniqueness part in Minkowski's Theorem \cite[Th. 7.2.1]{Sc}.
However, a priori the reflection that makes~\eqref{curvatures_locally_equal} valid may vary from component to component. 
An important difference between $n=2$ and $n\geq3$, something which has been used in overcoming this difficulty in $\R^2$, is the fact that in the plane \eqref{curvatures_locally_equal} implies that a portion of $\pa H$ is a translation of a portion of $\pa K$, while this is not true anymore in $\R^n$ when $n\geq3$.
Our extension of Kobayashi's result on the asymptotic behavior of $\cZ(K)$ is the key to prove that the reflection that makes~\eqref{curvatures_locally_equal} valid does not vary from component to component. In order to explain this we observe that~\eqref{convoluzione}, with $H=K$, implies
\[
\fou{g_K}(\ze)=\fou{1_K}(\ze)\conj{\fou{1_K}\left(\conj{\ze}\right)}
\]
(the bar denotes conjugation) because $\fou{1_{-K}}(\zeta)=\conj{\fou{1_{K}}\left(\cz\right)}$. Thus
\[
 \{\ze\in\C^n : \fou{g_K}(\ze)=0\}=\cZ(K)\cup\conj{\cZ(K)}.
\]
Determining $K$ from $g_K$ can be proved equivalent to resolve the ambiguity in determining  $\cZ(K)$ from $\cZ(K)\cup\conj{\cZ(K)}$. 
In this context the ambiguity in determining whether a reflection of $H$ is necessary  to make~\eqref{curvatures_locally_equal} valid is analogous to the ambiguity in determining, given any  component of $\cZ(K)\cap\{\ze\in\C^n : \im \ze\neq0\}$, whether it  is contained in $\cZ(H)$ or it is contained in  $\conj{\cZ(H)}$. 
The key ingredient in resolving this ambiguity, when the body is $C^\ds_+$ regular, is the fact that certain maps appearing in the description of the asymptotic behavior of $\cZ(K)$ are analytic.


We briefly remark that similar ideas enable us to prove also a result regarding the equivalent of the Covariogram Problem for the cross-covariogram, where one asks for the determination of a pair $(H,K)$ of convex bodies from $g_{H,K}$, up to the inherent ambiguities. Bianchi~\cite{B4} solves this problem in the class of pairs of convex polygons by completely classifying the pairs which are not determined and here we are able to prove that all pairs of $C^8_+$ regular planar convex bodies \emph{are determined}.
We point the reader to Section~\ref{sec_cross_cov} for the details.

The second connection between the Covariogram Problem and properties of the Fourier transform of $1_K$ in $\C^n$ comes from some results regarding the uniqueness aspects of the Phase Retrieval Problem. 
Let $f\in L^2(\R^n)$ have compact support. It is known that $\fou{f}$ is an entire function (i.e. is holomorphic on the entire $\C^n$).
Barakat and Newsam~\cite{Barakat-Newsam-1984}, Sanz and Huang \cite{Sanz-Huang-1984}, Stefanescu~\cite{Stefanescu-1985} and Hurt~\cite{Hurt-1989} prove that the nonuniqueness in the determination of $f$ from $|\fou{f}|$  is related to the possibility of factoring $\fou{f}$ as the product of two nontrivial entire functions. 
What is the significance of these results for the Covariogram Problem? 
An analogy is constituted by the fact that all known examples of nondetermination of $K$ from $g_K$ arise from the possibility of ``factoring'' $K$ as a Cartesian product of lower dimensional convex bodies contained in complementary subspaces. 
In this case  $\fou{1_K}$ can be written  indeed as the product of nontrivial entire functions.
\smallskip

\textbf{Problem 3.} \emph{Is it possible to find explicit geometric conditions on a convex body $K$ that grant that $\fou{1_K}$ cannot be factored?}
\smallskip

This seems to be a very difficult problem and we do not have results on it. 
To explain its difficulty let us observe that  the subproblem consisting in understanding for which $K$  the function $\fou{1_K}$ can be factored as the product of a polynomial $P(\ze)$ and of an entire function is equivalent to understanding for which $K$ the differential problem associated to $P$ has a solution with compact support $K$. 
When $P(\ze)=\ze_1^2+\dots+\ze_n^2-c$ this problem has been studied in many papers, because proving that a solution to this problem exists for some $c>0$ only if $K$ is a ball is equivalent to the Pompeiu Problem. 
In Section~\ref{sec_phase_retr} we explain all this in detail.

We conclude the introduction by mentioning another contribution to the Covariogram Problem proved using results of the theory of functions of several complex variables.  A natural question is the following one.
\smallskip

\textbf{Problem 4.} \emph{Is it possible to read in $g_E$ symmetry properties of the set $E$?}
\smallskip

Note that $g_E$ is always an even function, independently of any symmetry property of $E$. Lawton~\cite[Corollary~1]{Lawton-1981} implies the following result.
\begin{theorem}[Lawton~\cite{Lawton-1981}]\label{teo_radial_symmetry}
Let $n\geq 2$ and $E\subset\R^n$ be a  compact set which is the closure of its interior and assume that  $g_E$ is radially symmetric. Then a translation of $E$ is radially symmetric and $E$ is determined by $g_E$, up to translation  and reflection, in the class of compact sets which are the closure of their interior.
\end{theorem}
W.~Lawton proves the corresponding result for real-valued $L^2(\R^n)$ functions with compact support as a consequence of a representation formula for entire functions of exponential type such that the modulus of their restriction to $\R^n$ is radially symmetric and in $L^2(\R^n)$.

\section{Definitions, notations and preliminaries}\label{sec_definitions}
\subsection{Basic definitions and notation}
As usual, $\Sn$ denotes the unit sphere and  $o$ the
origin in the Euclidean $n$-space $\R^n$. If $x,y\in\R^n$, then $\left<x,y\right>$ is the scalar product of $x$ and $y$, while $|x|$ is the norm of $x$. 
If $\ze\in\C^n$ and $\ze=x+\ii y$, with $x,y\in\R^n$, then $\re\ze$ and $\im \ze$ denote respectively $x$ and $y$. Moreover $|\ze|=(|\re\ze|^2+|\im \ze|^2)^{1/2}$ denotes the norm of $z$. If $u\in\Sn$, then $u^{\perp}$ is the $(n-1)$-dimensional subspace orthogonal to $u$. 
For $\de>0$ and $x\in\R^n$, $B(x,\de)$ denotes $\{y\in\R^n : |y-x|<\de\}$. When $\ze\in\C^n$, $B(\ze,\de)$ is defined similarly. We write $\la_n$ for $n$-dimensional Lebesgue measure in $\R^n$. We define $\omega_n$ the surface area of the unit ball in $\R^n$.

For $t\in\R$ let $t_+=\max\{t,0\}$, $t_-=\max\{-t,0\}$ and let $[t]$ denote the integer part of $t$.

We denote by $\partial E$,  $\inte E$,  $\cl E$, and $1_E$ the {\it boundary}, {\it
interior}, \emph{closure}, and {\em characteristic function} of a set $E$ in $\R^n$, respectively. 
A set is {\it $o$-symmetric} if it is centrally symmetric, with center at the origin.
If $E$ and $F$ are sets in $\R^n$, then $E+F=\{x+y: x\in E, y\in F\}$ denotes their {\em Minkowski sum}.

Given a function $f$ defined on a subset of $\R^n$, $\supp f$, $\nabla f$ and $D^2 f$ denote its support, its gradient and its Hessian, respectively. We say that $f\in\C^\infty_0(\R^n)$ if $f$ is $m$-times differentiable for each $m\in\N$ and $\supp f$ is compact.

\subsection{Convex geometry and covariogram}
A {\em convex body} in $\R^n$ is a compact convex set with nonempty interior. The treatise of
Schneider \cite{Sc} is an excellent general reference for convex geometry.
The function
$$h_K(u)=\max\{\left<u,y\right>: y\in K\},$$
for $u\in\R^n$, is the {\it support function} of $K$ and
$$\width_K(u)=h_K(u)+h_K(-u),$$
its {\it width function}.  Any convex body $K$
is uniquely determined by its support function.

We say that a convex body $K$ is in the class $C^m$, for $m\in\N$, if it is a $m$-differentiable manifold.  We say that $K\in\C^m_+$, for $m\geq2$, if $K\in C^m$ and the \emph{Gauss curvature} of $\pa K$ is positive everywhere.    We say that  $K\in\C^\infty_+$ if $K\in C^m_+$ for each $m\in\N$. 

When $K\in\C^2_+$, $\nu_K:\pa K\to\Sn$ denotes the \emph{Gauss map} and $\gau_K(u)$ denotes the Gauss curvature of $\pa K$ at the point  $\nu^{-1}(u)$ on $\pa K$ with outer normal $u\in\Sn$.
Let $\Wein_K(u)$ denote the \emph{Weingarten map}, i.e. the differential of the {Gauss map} $\nu_K$ of $\pa K$ computed at $\nu^{-1}_K(u)$.
The eigenvalues of $\Wein_K(u)$ are the principal curvatures of $\pa K$ at $\nu^{-1}_K(u)$ and their product equals the Gauss curvature  $\gau_K(u)$.

The covariogram and the cross covariogram have been defined in the introduction.

Let $H$, $H'$, $K$ and $K'$ be convex bodies in  $\R^n$.
The translation of $H$ and $K$ by the same vector, and the substitution of $H$ with $-K$ and of $K$ with $-H$, leave $g_{H,K}$ unchanged.  We call $(H,K)$ and $(H',K')$ \emph{trivial associates} when one pair is obtained by the other one via a combination of the two operations above, that is, when either $(H,K)=(H'+x,K'+x)$ or $(H,K)=(-K'+x,-H'+x)$, for some $x\in \R^n$.

We have $g_{H,K}(x)=0$  if and only if $x\notin H+(-K)$, so the support of $g_{H,K}$ is $H+(-K)$. Since the support function is linear with respect to Minkowski addition we have
\begin{equation}\label{width_of_support}
\width_{\supp g_{H,K}}=\width_H+\width_K.
\end{equation}

\subsection{Fourier-Laplace  and Radon transform}
An \emph{entire function} is a complex-valued function that is holomorphic over the whole $\C^n$. An entire function $f$ is of \emph{exponential type} if there exist $a, b\in\R$ and $m\in\mathbb{Z}$ such that $|f(\ze)|\leq a(1+|\ze|)^m e^{b |\im \ze|}$, for each $\ze\in\C^n$.

The \emph{Fourier-Laplace transform} of a function $f\in L^2(\R^n)$ with compact support is defined for $\zeta\in\C^n$ as
\[
\fou{f}(\zeta)=\int_{{\R^n}}e^{\ii \left<x,\zeta\right>}f(x)\ dx.
\]
By the Paley-Wiener Theorem $\fou{f}$ is an entire function of exponential type whose restriction to $\R^n$ belongs to $L^2$ if and only if $f\in L^2(\R^n)$ and has compact support. The version of this theorem for distributions asserts that $\fou{f}$ is an entire function of exponential type  if and only if $f$ is a distribution with compact support. See \cite[Theorem~7.23]{Rudin-91}.   Distributions will enter this paper only very marginally and we refer to Rudin~\cite{Rudin-91} for their definition.

Taking Fourier transforms in \eqref{convoluzione} and using the identity
\begin{equation}\label{riflessione_in_Cn}
\fou{1_{-K}}(\zeta)=\conj{\fou{1_{K}}\left(\cz\right)},
\end{equation}
valid for every $\ze\in\C^n$, we obtain the relation
\begin{equation}\label{convoluzione_in_cn}
\fou{g_K}(\zeta)
=\fou{1_K}(\zeta)\,\conj{\fou{1_{K}}\left(\cz\right)}.
\end{equation}

Given a convex body $K$ in $\R^n$, $t\in\R$ and $u\in \Sn$,  we denote by $S_K(u,t)$ the \emph{Radon transform} of $1_K$
\begin{equation*}
S_K(u,t)=\la_{n-1}\left(K\cap(u^\perp+t)\right).
\end{equation*}

\section{Some information that is easy to read in the covariogram of a $C^2_+$ convex body} \label{sec_c2plus}
This section is devoted to the following result.
\begin{proposition}\label{teo_information_c2+_bodies}
Let $K$ be a convex body of class $C^2_+$ in $\R^n$, $n\geq 2$, let $u\in\Sn$ and $p=\nu^{-1}_K(u)-\nu^{-1}_K(-u)$.
\begin{enumerate}[(I)]
\item\label{prop_comport_asint_covario} 
The knowledge of $g_K$ in a neighborhood of $p$ determines 
\[
\Wein_K(u)^{-1}+\Wein_K(-u)^{-1}\quad\text{and}\quad \det\left(\Wein_K(u)+\Wein_K(-u)\right).
\]
In particular, it determines
\begin{equation}\label{product_gauss_curv}
 \gau_K(u)\gau_K(-u).
\end{equation}
\item\label{prop_sum_radii_curvature}
The knowledge of $g_K$ in a neighborhood of $o$ determines
\begin{equation}\label{equal_sum_radii_curvature}
\frac1{\gau_K(u)}+\frac1{\gau_K(-u)}.
\end{equation}
\item\label{prop_sum_reverse_weingarten} 
When $n=2$ the width function $\width_K$ determines the expression in \eqref{equal_sum_radii_curvature}.
%
\item\label{nonordered_curvatures}
The covariogram $g_K$ determines $\{\gau_K(u),\gau_K(-u)\}$.
\end{enumerate}
\end{proposition}

The point $p$ in Assertion~\eqref{prop_comport_asint_covario} is the point of the boundary of $\supp g_K$ with outer normal $u$, by the identity $\supp g_K=K+(-K)$ and  \cite[Th. 1.7.5(c)]{Sc}. Studying the  behavior of $g_K$ near $p$ is equivalent to studying the behavior of the volume of $K\cap(K+x)$ for $x$ such that $K\cap(K+x)$ is contained in a small neighborhood of $\nu^{-1}_K(u)$. For these $x$ the boundary of $K\cap(K+x)$ consists of a portion of  $\pa K$ near $\nu^{-1}_K(u)$ and of (a translation of) a portion of $\pa K$ near $\nu^{-1}_K(-u)$.
Regarding Assertion~\eqref{prop_sum_reverse_weingarten} we recall (see \eqref{width_of_support}) that knowing $\width_K$ is equivalent to knowing $\supp g_K$.

The next lemma is needed to prove Assertion~~\eqref{prop_comport_asint_covario}.

\begin{lemma}\label{lem_comport_asintot_covario}
Let $A$, $B$ be symmetric $(n-1) \times (n-1)$ positive-definite matrices. Let $t\in\R$, $t>0$, and  $q\in\R^{n-1}$ be such that $2t-\left<\left(A^{-1}+B^{-1}\right)^{-1}q, q\right>\geq0$.
Let $f_1$, $f_2:\R^{n-1}\to\R$  be the quadratic functions
\[
f_1(x)=t-\frac{1}{2}\left<A(x-q),x-q\right>,\quad f_2(x)=\frac{1}{2}\left<Bx,x\right>.
\]
 Then the volume  of the region in $\R^{n}$ bounded by the graphs of $f_1$ and $f_2$  is
\begin{multline*}
 \la_n\left\{(x,x')\in\R^{n-1}\times\R : f_2(x)\leq x'\leq f_1(x)\right\}\\
 =\frac{\omega_{n-1} 2^{(n+1)/2}}{n^2-1}
 \frac{ \left( 2t-\left<\left(A^{-1}+B^{-1}\right)^{-1}q, q\right>\right)^{(n+1)/2}}
{\sqrt{\det(A+B)}}.
\end{multline*}
\end{lemma}
\begin{proof}
We have
\begin{equation}\label{espressione1}
f_1(x)-f_2(x)=t-\frac{1}{2}\Big( \left<(A+B)x, x\right>-\left<Ax, q\right>-\left<Aq, x\right>+\left<Aq, q\right> \Big).
\end{equation}
Let us consider the expression  in parentheses in the right hand side of \eqref{espressione1}. By adding and subtracting $\left<Aq, (A+B)^{-1}Aq\right>$, by rewriting $\left<Ax, q\right>$ as $\left<(A+B)x,(A+B)^{-1}Aq\right>$ (a consequence of the symmetry of $A$ and $B$) and by regrouping some terms, we  obtain
\begin{multline}\label{espressione2}
 \left<(A+B)x, x\right>-\left<Ax, q\right>-\left<Aq, x\right>+\left<Aq, q\right>=\\
 =\left<(A+B)y, y\right>+\left<\left(A-A(A+B)^{-1}A\right)q, q\right>,
\end{multline}
where $y=x-(A+B)^{-1}Aq$. The identity
\begin{equation}\label{woodbury}
A-A(A+B)^{-1}A=\left(A^{-1}+B^{-1}\right)^{-1}
\end{equation}
is a special case of the Woodbury matrix identity \cite{Henderson-Searle-1981}. Formulas~\eqref{espressione1}, \eqref{espressione2} and \eqref{woodbury} imply
\[
  f_1(x(y))-f_2(x(y))=s-\frac1{2}\left<(A+B)y, y\right>,
\]
where  $s=t-(1/2)\left<\left(A^{-1}+B^{-1}\right)^{-1}q, q\right>$.

Let $V(q,t)$ denote the volume that we wish to compute. It is
\begin{align*}
 V(q,t)
=&\int_{\{x\in\R^{n-1} : f_1(x)-f_2(x)\geq0\}}\big(f_1(x)-f_2(x)\big)\ dx\\
=&\int_{\{y\in\R^{n-1} : 2s-\left<(A+B)y, y\right>\geq0\}}\left(s-\frac1{2}\left<(A+B)y, y\right>\right)\ dy.
\end{align*}
Since  an orthogonal transformation does not change $V(q,t)$, we may assume that the symmetric matrix $A+B$ is diagonal. Let $\la_i$ be the $i$-th element of the diagonal of $(A+B)$ and let $w=(\sqrt{\la_1}y_1,\dots,\sqrt{\la_{n-1}}y_{n-1})$. We have
\begin{align*}
 V(q,t)
=&\frac1{\sqrt{\det(A+B)}}\int_{\{w\in\R^{n-1}:|w|^2\leq 2s\}}s-\frac{|w|^2}{2}\ dw\\
=&\frac{\omega_{n-1}}{\sqrt{\det(A+B)}}\int_0^{\sqrt{2s}}r^{n-2}\left(s-\frac{r^2}2\right)\,dr\\
=&\frac{\omega_{n-1} 2^{(n+1)/2}}{n^2-1} \frac{s^{(n+1)/2}}{\sqrt{\det(A+B)}}.
\end{align*}
Writing $s$ in terms of $q$ and $t$ concludes the proof.
\end{proof}

\begin{proof}[Proof of Proposition~\ref{teo_information_c2+_bodies}]

\emph{Assertion~\eqref{prop_comport_asint_covario}.} Let us compute the  asymptotic expansion of $g_K$ near $p$. Changing, if necessary, the coordinate system we may assume  $u=(0,\dots,0,-1)$ and $\nu_K^{-1}(u)=o$. 
Let $\nu_K^{-1}(-u)=(a,s)\in\R^{n-1}\times\R$. We have $p=-(a,s)$. Let $A$ and $B$ be the matrices representing respectively $\Wein_K(-u)$ and $\Wein_K(u)$ in an orthonormal basis in $u^\perp=\R^{n-1}$.

We prove that when  $q\in\R^{n-1}$ and $t>0$ are such that
\begin{equation}\label{condition_t_and_q}
2t-\left<\left(A^{-1}+B^{-1}\right)^{-1}q, q\right>>0,
\end{equation}
we have
\begin{multline}\label{svil_asintot}
 g_K\big(q-a,t-s\big)=\\
 =\frac{\omega_{n-1}2^{(n+1)/2}}{n^2-1}
 \frac{ \left( 2t-\left<\left(A^{-1}+B^{-1}\right)^{-1}q, q\right>\right)^{(n+1)/2}}
{\sqrt{\det(A+B)}}
\left(1+\epsilon(q,t)\right),
\end{multline}
where 
\[
 \lim_{\substack{(q,t)\to0\\ \text{\eqref{condition_t_and_q} holds true}}} \epsilon(q,t)=0.
\]
The boundary of $K\cap (K+(q-a,t-s))$, the set whose volume is measured by $g_K\big(q-a,t-s\big)$, consists of a portion of $\pa K$ near $(a,s)$ translated by the vector $(q-a,t-s)$ and of a portion of $\pa K$ near $o$.
Since $K$ is sufficiently smooth, $\partial K$ can be approximated in a neighborhood of $(a,s)$ (up to terms of higher order) by the graph of the paraboloid $\{(x,x')\in\R^{n-1}\times\R : x'=s-\left<A(x-a),(x-a)\right>\}$.
Similarly, $\partial K$ can be approximated in a neighborhood of $o$ (up to terms of higher order) by the graph of the paraboloid $\{(x,x')\in\R^{n-1}\times\R : x'=\left<Bx, x\right>\}$.

For $q\in\R^{n-1}$ and $t>0$ sufficiently small and satisfying~\eqref{condition_t_and_q}, $K\cap (K+(q-a,t-s))$ is contained in
\[
 \{(x,x'): \left<B_1(q,t)x, x\right>\le x'\le t-\left<A_1(q,t)(x-q), x-q\right> \}
\]
and contains
\[
 \{(x,x'): \left<B_2(q,t)x, x\right>\le x'\le t-\left<A_2(q,t)(x-q), x-q\right> \},
\]
where $A_1(q,t)$, $A_2(q,t)$, $B_1(q,t)$ and $B_2(q,t)$ are symmetric positive-definite matrices such that
$A_1(q,t)<A <A_2(q,t)$, $B_1(q,t)<B <B_2(q,t)$,
\[
 \lim_{q\rightarrow o\,,t\rightarrow 0^+} A_1(q,t)=\lim_{q\rightarrow o\,,t\rightarrow 0^+} A_2(q,t)=A,
\]
and
\[
 \lim_{q\rightarrow o\,,t\rightarrow 0^+} B_1(q,t)=\lim_{q\rightarrow o\,,t\rightarrow 0^+} B_2(q,t)=B.
\]
By Lemma~\ref{lem_comport_asintot_covario} the volume  of these two sets is
\[
\frac{\omega_{n-1}2^{(n+1)/2}}{n^2-1}
\frac{ \left( 2t-\left<\left(A_i^{-1}+B_i^{-1}\right)^{-1}q, q\right>\right)^{(n+1)/2}}
{\sqrt{\det(A_i+B_i)}},
\]
$i=1,2$. Since the difference between the previous expression for $i=1$ and that for $i=2$ is $\littleO(2t+|q|^2)$, as $q$ tends to $o$ and $t$ tends to $0^+$, we have \eqref{svil_asintot}. This asymptotic expansion proves the first claim of the proposition.

To prove the last claim it suffices to observe that $\gau_K(u)=\det B$, $\gau_K(-u)=\det A$ and that the knowledge of $\left(A^{-1}+B^{-1}\right)$ and of $\det(A+B)$ gives $\det A \det B $, since  $\det\left(A^{-1}+B^{-1}\right)\det A \det B =\det(A+B)$.

\emph{Assertion~\eqref{prop_sum_radii_curvature}.} Matheron~\cite[p.~86]{Matheron-1975}  proves that for each $v\in\Sn$ we have
\[
\frac{\partial^+ g_K}{\partial v}(o)=-\la_{n-1}\left(K|v^\perp\right),
\]
where $\partial^+ /\partial v$ denotes left directional derivative in direction $v$, and $K|v^\perp$ denotes  the orthogonal projection of $K$ on $v^\perp$.
\cite[Theorem 3.3.2]{Gar95ed2} proves that the knowledge of $\la_{n-1}\left(K|v^\perp\right)$ for each $v\in\Sn$ determines  the expression in~\eqref{equal_sum_radii_curvature}. 

\emph{Assertion~\eqref{prop_sum_reverse_weingarten}.} This is an immediate  consequence of Theorems 3.3.2 and 3.3.5 in~\cite{Gar95ed2}.

\emph{Assertion~\eqref{nonordered_curvatures}.} The  expressions in~\eqref{equal_sum_radii_curvature} and~\eqref{product_gauss_curv} determine $\{\gau_K(u), \gau_K(-u)\}$.
\end{proof}

\section{Proof of Kobayashi result under lower regularity assumption} \label{sec_kobayashi_cm}
Let 
 \[
S=\{z\in\C^n : z=\ze u, \text{ with }\ze\in\C, u\in \Sn\}.
 \]
In $S$ we identify $\ze u$ and $(-\ze)(-u)$, for each $\ze\in\C$ and $u\in\Sn$. 
Let
\[
\cZ(K)=\{\zeta\in\C^n\ :\ \fou{1_K}(\zeta)=0\}.
\]

\begin{theorem}[T.~Kobayashi \cite{Kob1}]\label{teo_kobayashi} Let $S$ be defined as above. Let $K$ be a convex body in $\R^n$ of class  $C^\infty_+$. Then there exists  a positive integer $m(K)$ such that 
\[
 \cZ(K)\cap S=\left( \bigcup_{m=m(K)}^\infty\cZ_m(K) \right)\bigcup C(K),
\]
where the union is disjoint, $C(K)$ is a bounded set and, for each integer $m\geq m(K)$, $\cZ_m(K)$ is analytically diffeomorphic to $\Sn$.
More precisely, for each integer $m\geq m(K)$, there exists an analytic map $F_{m,K}:\Sn\to\C$ such that
\begin{equation}\label{rappresentazione_Zm_con_Fm}
 \cZ_m(K)=\{ F_{m,K}(u)\,u\ : u\in\Sn\},
\end{equation}
we have
\begin{equation}\label{rappresentazione_mappa_analitica}
F_{m,K}(u)=\frac{\pi(4m+n-1)}{2\width_K(u)}+\ii\ \frac{\ln\gau_K(-u)-\ln\gau_K(u)}{2\width_K(u)}+\bigO\left(\frac1{m}\right),
\end{equation}
and the error term $\bigO(1/m)$ in \eqref{rappresentazione_mappa_analitica} tends to $0$, as $m$ tends to infinity, uniformly with respect to $u\in \Sn$.
\end{theorem}

We are interested in lowering the  regularity assumption on $K$ needed for the conclusions of Theorem~\ref{teo_kobayashi} to hold. We are able to prove the following result.

\begin{theorem}\label{teo_kobayashi_cm}Let $\ds$ be as in Theorem~\ref{teo_cov_smooth}. If the convex body $K$ in $\R^n$ is of class $C^\ds_+$ then the conclusions of Theorem~\ref{teo_kobayashi} hold.
\end{theorem}

We remark that the regularity assumption in Theorem~\ref{teo_kobayashi_cm} is analogous to that required in some studies of the asymptotic behavior at infinity of $\fou{1_K}$ in $\R^n$ (see, for instance, Herz~\cite{Herz-1962}).

The proof of Theorem \ref{teo_kobayashi} is presented  both in \cite[Theorem 2.3.6]{Kob1} and in \cite{Kob2}. 
The Fourier-Laplace transform $\fou{1_K}(\ze u)$, for $\ze\in\C$ and $u\in\Sn$, is written as the Fourier-Laplace transform of the Radon transform $S_K(u,t)$ with respect to the single variable $t$, i.e.
\begin{equation}\label{radon_plus_fourier_oned}
\fou{1_K}(\ze u)=\int_{-\infty}^{\infty}S_K(u,t)e^{\ii t\ze}\,dt.
\end{equation}
Some results proved in \cite{Kob2} and regarding the zero set of the Fourier-Laplace transform of functions of a single variable are then applied to this expression.
We refer in particular to \cite[Corollary 2.20]{Kob2}, which gives the asymptotic behavior at infinity of the zeros of the Fourier-Laplace transform of a function and an estimate on the dimensions of a compact set containing the remaining zeros.
It is the application of this corollary which yields the conclusions of Theorem~\ref{teo_kobayashi}, and both \cite[Lemma 3.14]{Kob2} and \cite[Lemma 2.2.8]{Kob1} prove that  $S_K(u,\cdot)$ satisfies the assumptions of this corollary when $K$ belongs to $C^\infty_+$.
The next lemma proves that $S_K(u,\cdot)$ satisfies the assumptions of \cite[Corollary 2.20]{Kob2} also when $K\in C^{\ds}_+$. 

Let $\psi(x):\R\to[0,1]$ be a $C^\infty$ function such that $\supp \psi\subset[-2,2]$ and  $\psi(x)\equiv1$ when $x\in[-1,1]$.
\begin{lemma}\label{lem_kobayashi_cm}
Let $\ds$ be as in Theorem~\ref{teo_cov_smooth} and let $K\subset\R^n$ be a convex body of class $C^\ds_+$. Let 
\[
V=\{(u,t)\in\Sn\times\R: -h_K(-u)<t<h_K(u)\}
\]
and, for $u\in\Sn$, let
\begin{gather*}
a_0(u)=\frac{(2\pi)^{\frac{n-1}2}}{\Gamma(\frac{n+1}2)\sqrt{\gau_K(-u)}},\quad\quad
b_0(u)=\frac{(2\pi)^{\frac{n-1}2}}{\Gamma(\frac{n+1}2)\sqrt{\gau_K(u)}},
\end{gather*}
and
\[
\phi(u,t)=\psi\left(\frac{5t}{\width_K(u)}\right).
\]
Then the following assertions hold:
\begin{enumerate}[(I)]
\item\label{ass:lem_kobayashi_cm_I} The Radon transform $S_K$ is continuous in $\Sn\times\R$ and its support is $\cl(V)$. Moreover $S_K$ is differentiable $\ds$ times with respect to $t$ at every $(u,t)\in V$ and each of these derivatives is continuous in  $t$ and in $u$;
\item\label{ass:lem_kobayashi_cm_II} For every $u\in\Sn$ there exist $a_1(u), a_2(u), b_1(u), b_2(u)\in\R$ such that
\begin{multline}\label{taylor_radon}
 S_K(u,t)-\sum_{j=0}^2a_j(u)\Big(t+h_K(-u)\Big)^{\frac{n-1}2+j}_+
\phi\big(u,t+h_K(-u)\big)+\\-\sum_{j=0}^2b_j(u)\Big(t-h_K(u)\Big)^{\frac{n-1}2+j}_-
\phi\big(u,t-h_K(u)\big),
\end{multline}
as a function of $t$, belongs to $C^{\left[(n-1)/2\right]+2}(\R)$, and its derivative of order $\left[(n-1)/2\right]+3$ exists in $(-h_K(-u), h_K(u))$.
\item\label{ass:lem_kobayashi_cm_III} The expressions  $|a_1(u)|$, $|a_2(u)|$, $|b_1(u)|$, $|b_2(u)|$ and 
\[
 \sup_{-h_K(-u)< t< h_K(u)}\left|\frac{\pa^{\frac{n-1}{2}+3}\big(\text{expression in \eqref{taylor_radon}}\big)}{\pa t^{\frac{n-1}{2}+3}}\right|
\]
are  bounded from above uniformly with respect to $u$ in $\Sn$.
\end{enumerate}
\end{lemma}
\begin{proof}
\emph{Assertion~\eqref{ass:lem_kobayashi_cm_I}.} The claims regarding the continuity and the support of $S_K$ are obvious.  
The claim regarding the derivatives is essentially proved in \cite[Lemma~2.4]{Kol05}. This lemma proves that when $K$ is $o$-symmetric then $S_K$ is differentiable $\ds$ times with respect to $t$ at every $(u,t)$ such that $|t|$ is sufficiently small, and at  $(u,t)$ each of these derivatives is continuous in  $t$ and in $u$. However an inspection of the proof of this lemma easily shows that the $o$-symmetry of $K$  is not needed, as the author confirms~\cite{Koldob-personal}. 
Thus, let $(u_0,t_0)\in V$ and let $x_0\in\inte K\cap (u_0^\perp+t_0)$. Note that $\inte K\cap (u_0^\perp+t_0)\neq\emptyset$, by definition of $V$. Let us apply \cite[Lemma~2.4]{Kol05} with $x_0$ playing the role of the origin. It proves that   the function
\[
(u,s)\to\la_{n-1}\left(K\cap\left\{x : \left<x-x_0,u\right>=s\right\}\right)
\]
is differentiable $\ds$ times with respect to $s$, and each of these derivatives is continuous in  $s$ and in $u$ whenever $|s|$ is sufficiently small.
The previous function coincides with $S_K(u,s+\left<x_0,u\right>)$. This implies the  requested property of $S_K(u,t)$  at each $(u,t)$ such that $|\left<x_0,u\right>-t|$ is sufficiently small, that is in a neighborhood of  $(u_0,t_0)$.

\emph{Assertion~\eqref{ass:lem_kobayashi_cm_II}.}
Since the expression in \eqref{taylor_radon}, as a function of $t$, vanishes outside $(-h_K(-u), h_K(u))$, it  belongs to $C^\ds$ in that interval,  and $\ds\geq[(n-1)/2]+3$, it suffices to prove the assertion in a neighborhood of each endpoint of that interval. We will only do it in a neighborhood of $h_K(u)$, since the proof for the other endpoint is similar.
 
Let $u_0\in \Sn$, let $U\subset\Sn$ be a neighborhood of  $u_0$ and let $u\in U$. Let $e_1(u),\dots,e_{n-1}(u)\in \R^n$ denote an  basis of $u^\perp$ which is a $C^\infty$ function of $u$, and for $y=(y_1,\dots,y_{n-1})\in\R^{n-1}$ let $L(u,y)=\sum_{i=1}^{n-1}y_i e_i(u)$. 
For each $u\in U$ we parametrize $\pa K$ in a neighborhood $W(u)$ of $\nu^{-1}_K(u)$ as 
\begin{equation}\label{lem_kobayashi_cm_aa}
 \pa K\cap W(u)=\{\nu^{-1}_K(u)+L(u,y)-f(u,y) u : y\in V\},
\end{equation}
where $V\subset\R^{n-1}$ is a suitable neighborhood of $o$ and $f(u,y)$ is defined implicitly by~\eqref{lem_kobayashi_cm_aa}. 
This is equivalent to saying that in a Cartesian coordinate system $(y_1,\dots,y_n)$ whose origin is $\nu^{-1}_K(u)$ and whose positive $y_n$-semiaxis points in the direction of $-u$ the surface $\pa K\cap W(u)$ is the graph of the convex function $y_n=f(u,y_1,\dots,y_{n-1})$. For each $(u,y)\in U\times V$, $f(u,y)$ is nonnegative and convex with respect to $y$. We have
\[
 f(u,o)=0,\quad\nabla_y f(u,o)=0
\]
and the eigenvalues of $D^2_y f(u,o)$ are the principal curvatures of $\pa K$ at $\nu_K^{-1}(u)$. Since $K$ is of class $C^\ds_+$ the map $\nu^{-1}_K(u)$ belongs to $C^{\ds-1}(S^{n-1})$. Therefore $f$ and $\nabla_y f$ belong to $C^{\ds-1}(U\times V)$.

Now let us express the Radon transform $S_K(u,t)$, for $t$ close to $h_K(u)$,  in terms of $f$. When $t>h_K(u)$ the set $K\cap(u^\perp+t)$ is empty, while when $t\leq h_K(u)$ we have
\begin{multline}\label{lem_kobayashi_cm_expression_of_section}
K\cap(u^\perp+t)=
 \Big\{\nu^{-1}_K(u)+L(u,y)-(h_K(u)-t) u :\\
 \text{ $y\in\R^{n-1}$ satisfies }f(u,y)\leq h_K(u)-t\Big\}.
\end{multline}
We are thus interested in  the measures of the level sets $\{y\in\R^{n-1}: f(u,y)\leq s\}$ for small positive values of $s$. 
Let us start by expressing $f(u,\cdot)$ in polar coordinates $(r,\te)$. For reasons that will be clear in a few lines, we  let the parameter $r$ free to take also negative values. More precisely, for $(u,r,\te)\in U\times(-\ee,\ee)\times \Sndue$, for a sufficiently small $\ee>0$, we define $f_0(u,r,\te)=f(u,r\te)$.
The properties of $f$ imply that there is a continuous function $f_1$ such that
\[
f_0(u,r,\te)=r^2f_1(u,r,\te).
\]
Note that $f_1$ and $\nabla_{(r,\te)}f_1$ belong to $C^{\ds-3}\left(U\times(-\ee,\ee)\times\Sndue\right)$.
Since $f_1(u,0,\te)>0$, for each $(u,\te)\in U\times \Sndue$, after possibly changing $U$ and $\ee$, we may assume $f_1(u,r,\te)>0$ in $U\times(-\ee,\ee)\times\Sndue$.
By the Implicit Function Theorem there exists $U'\subset U$ neighborhood of $u_0$, there exist $\de>0$ and a function $R:U'\times (-\de,\de)\times\Sndue\to(-\ee,\ee)$ such that
\[
 R(u,t,\te)\sqrt{f_1\left(u,R(u,t,\te),\te\right)}=t.
\]
The regularity of $f_1$ and the fact that $R=0$ if and only $t=0$ imply that $R$ and $\nabla _{(t,\te)}R$ belong to $C^{\ds-3}\left(U'\times(-\de,\de)\times\Sndue\right)$.  The geometrical meaning of $R$ is the following: when $t\in(-\de,\de)$ we have
\[
 \{y=R(u,t,\theta)\theta : \theta\in S^{n-2}\}=\{y\in\R^{n-1}: f(u,y)=t^2\}.
\]
The sign of $R(u,t,\theta)$ coincides with the sign of $t$ and, since $f_1(u,r,\te)=f_1(u,-r,-\te)$, we have
$-R(u,t,\te)\sqrt{f_1\left(u,-R(u,t,\te),-\te\right)}=-t$. This implies
\begin{equation}\label{lem_kobayashi_cm_a}
-R(u,t,\te)=R(u,-t,-\te).
\end{equation}

We are now ready to explictly express $S_K(u,t)$ in terms of $R$. Let $u\in U'$ and $t\in(h_K(u)-\de^2,h_K(u))$. By formula~\eqref{lem_kobayashi_cm_expression_of_section} we have
\begin{align*}
 S_K(u,t)=&\la_{n-1}\left( \{y\in\R^{n-1} : f(u,y)\leq h_K(u)-t\} \right)\\
 =&\la_{n-1}\left( \{y=r\theta : \theta\in S^{n-2}, 0\leq r\leq R(u,\sqrt{h_K(u)-t},\theta)\} \right).
\end{align*}
Therefore for $u\in U'$ we have
\begin{equation*}
 S_K(u,t)=
\begin{cases}
           0&\text{when $t\geq h_K(u)$;}\\
	  \displaystyle \int_{\Sndue} \frac{R\left(u,\sqrt{h_K(u)-t},\te\right)^{n-1}}{n-1}\,d\te& \text{when $t\in(h_K(u)-\de^2,h_K(u))$,}
\end{cases}
\end{equation*}
where $d\te$ denotes $(n-2)$-dimensional Hausdorff measure. 

In order to  study the behavior of $S_K$  let  us write the Taylor expansion of $R(u,t,\theta)$ in $t$ at $t=0$. In order to simplify the notations we set $m=\ds-2$ and we omit to explicitly write the dependence of $R$, and of some other functions, on $u$ and on $\te$.
We can write the Taylor expansion of $R$ as
\begin{equation}\label{taylor_exp_R}
 R(t)=\sum_{i=1}^{m}c_i t^i+r(t),
\end{equation}
for suitable coefficients $c_i=c_i(u,\te)$ (which depend continuously on $u$) and with the remainder  $r(t)=r(u,t,\te)$ written as
\[
r(t)=\int_0^t\left(\frac{\pa^mR(s)}{\pa s^m}-\frac{\pa^mR(0)}{\pa s^m}\right)\frac{(t-s)^{m-1}}{(m-1)!}\, ds.
\]
For $k=0,\dots,m$ and $t\in(0,\de)$, it is easy to derive from the previous expression of $r$ the following bounds
\begin{equation}\label{stima_resto_a}
\left| \frac{\pa^k r(t)}{\pa t^k}\right|\leq \sup_{s\in[0,t]} \left|\frac{\pa^mR(s)}{\pa s^m}-\frac{\pa^mR(0)}{\pa s^m}\right|\ t^{m-k}.
\end{equation}
Let us prove that, for $j$ positive integer, $k=0,\dots,m$  and $t\in(h_K-\de^2,h_K)$, we have
\begin{equation}\label{stima_resto}
 \left|\frac{\pa^k \left(r\left(\sqrt{h_K-t}\right)\right)^j}{\pa t^k}\right|\leq
 d_{j,k} \left(\sup_{s\in[0,t]} \left|\frac{\pa^mR(s)}{\pa s^m}-\frac{\pa^mR(0)}{\pa s^m}\right|\right)^j(h_K-t)^{\frac{m j}2-k}.
\end{equation}
for a suitable positive constant $d_{j,k}$ which depends only on $j$ and $k$.
Indeed, using  \cite[Formula $3_n$]{McK56} to express the $k$-th derivative of a composite function, we have 
\begin{align*}
\frac{\pa^k r\left(\sqrt{h_K-t}\right)}{\pa t^k}=&
\sum_{i=1}^k \frac{\pa^i r(s)}{\pa s^i}|_{s=\sqrt{h_K-t}}\sum_{j=0}^i \frac{(-1)^{i-j}}{j! (i-j)!}\left(h_K-t\right)^\frac{i-j}2\frac{\pa^k \left(h_K-t\right)^\frac{j}2}{\pa t^k}\\
=&\sum_{i=1}^k  \frac{\pa^i r(s)}{\pa s^i}|_{s=\sqrt{h_K-t}}(h_K-t)^{\frac{i}2-k}\sum_{j=0}^i \frac{(-1)^{i-j+k}}{j! (i-j)!}\times\\
&\quad\quad\quad\quad\times\frac{j}2\left(\frac{j}2-1\right)\dots\left(\frac{j}2-k+1\right).
\end{align*}
This formula and \eqref{stima_resto_a} prove~\eqref{stima_resto} when $j=1$. In order to prove~\eqref{stima_resto} when $j>1$ it suffices to use \cite[Formula $9_n$]{McK56} to express the derivative in the left-hand side of~\eqref{stima_resto} in terms of derivatives of $r(\sqrt{h_K-t})$ and to use~\eqref{stima_resto} with $j=1$. We omit the details.


Let us now apply all these estimates to our case. For $t\in(h_K-\de^2,h_K)$ we write
\begin{equation*}\label{sviluppo_asintotico_SK}
\begin{split}
S_K(u,t)=& 
\frac1{n-1}\int_{\Sndue} \left(\sum_{i=1}^{m}c_i (h_K-t)^{i/2}+r\left(\sqrt{h_K-t}\right)\right)^{n-1}\,d\te\\
=&\frac1{n-1}\int_{\Sndue} \left(\sum_{i=1}^{m}c_i (h_K-t)^{i/2}\right)^{n-1}\,d\te+\\
&+\frac1{n-1}\int_{\Sndue}\sum_{j=1}^{n-1}\binom{n-1}{j} \left(r\left(\sqrt{h_K-t}\right)\right)^j\left(\sum_{i=1}^{m}c_i (h_K-t)^{i/2}\right)^{n-1-j}\,d\te.
\end{split}\end{equation*}
Let $I_1(t)=I_1(u,t)$ and $I_2(t)=I_2(u,t)$ denote respectively the first and the second integral after the last equality sign in the previous formula.

First we study $I_1(t)$. This integral can be written as 
\begin{equation}\label{first_integral}
I_1(t)= \frac1{n-1}\sum_{l=(n-1)}^{m(n-1)}(h_K-t)^{l/2}\int\limits_{\Sndue}\sum_{\substack{i_1,\dots,i_{n-1}=1,\dots,m\\i_1+\dots+i_{n-1}=l}}c_{i_1}\dots c_{i_{n-1}}\,d\te.
\end{equation}
 Formula~\eqref{lem_kobayashi_cm_a} implies, for $i=1,\dots,m$,
\begin{equation}\label{proprieta_c_i}
 c_i(u,-\te)=(-1)^{i-1} c_i(u,\te).
\end{equation}
Let us prove that when $l-(n-1)$ is odd  the integrand in~\eqref{first_integral} is an odd function of $\te$. Indeed   let us write $i_j=1+p_j$, for $j=1,\dots,n-1$. The integer $p_j$ varies from $0$ to $m-1$ and $p_1+\dots+p_{n-1}=l-(m-1)$. If $l-(m-1)$ is odd then an odd number of $p_j$ is odd, i.e., an odd number of $c_j$ is even. This fact, by~\eqref{proprieta_c_i}, implies that the integrand in~\eqref{first_integral} is  odd.   

A consequence of this is that when $l-(n-1)$ is odd  the coefficient of $(h_K-t)^{l/2}$ in $I_1(t)$ vanishes, and that $I_1(t)$ is a finite sum of powers of $h_K-t$ with exponent ${{(n-1)/2}+j}$, where $j$ is a nonnegative integer.
 It is known (see \cite{Kob1}) that the coefficient of $(h_K-t)^{(n-1)/2}$ in $I_1(t)$ is $b_0(u)$. Let $b_1(u)$ and $b_2(u)$ denote respectively the coefficients of $(h_K-t)^{{(n-1)/2}+1}$ and of $(h_K-t)^{{(n-1)/2}+2}$.
 When $t\in(h_K-\de^2,h_K)$ then
\begin{equation}\label{Iunomenoprimitermini}
 I_1(t)-b_0(u)(t-h_K)_-^\frac{n-1}{2}-b_1(u)(t-h_K)_-^{\frac{n-1}{2}+1}-b_2(u)(t-h_K)_-^{\frac{n-1}{2}+2}
\end{equation}
is a linear combination of powers of $h_K-t$ with exponents higher than or equal to $(n-1)/2+3$. Thus if we extend the definition of $I_1(t)$ to $(h_K-\de^2,h_K+\de^2)$ by putting $I_1(t)=0$ when $t\in[h_K,h_K+\de^2)$, the expression in \eqref{Iunomenoprimitermini}   belongs to $C^{\left[(n-1)/2\right]+2}(h_K-\de^2,h_K+\de^2)$. Moreover, when $t\in(h_K-\de^2,h_K)$,  the  $[(n-1)/2]+3$ derivative with respect to $t$ of the expression in \eqref{Iunomenoprimitermini}  is equal to
\[
 \sum_{l=(n-1)+6}^{m(n-1)}e_l(h_K-t)^{l/2-\left[\frac{n-1}{2}\right]-3} \int\limits_{\Sndue}\sum_{\substack{i_1,\dots,i_{n-1}=1,\dots,m\\i_1+\dots+i_{n-1}=l}}c_{i_1}\dots c_{i_{n-1}}\,d\te,
\]
for suitable constants $e_l$ depending only on $n$ and $l$. Since all powers of $h_K-t$ in this derivative have nonnegative exponents, its absolute value is uniformly bounded in $(h_K-\de^2,h_K)$. Since the coefficients $c_1,\dots, c_m$ depend continuously on $u$, this bound is locally uniform with respect to $u$.

Now we study  $I_2(t)$. This function is a linear combination of terms of the form
\begin{equation*}
(h_K-t)^{l/2}\int_{\Sndue}c_{i_1}\dots c_{i_{n-1-j}}\left(r\left(\sqrt{h_K-t}\right)\right)^j\,d\te,
\end{equation*}
with $j=1,\dots,n-1$, $l=n-1-j,\dots,m(n-1-j)$, $i_1,\dots,i_{n-1-j}=1,\dots,m$, $i_1+\dots+i_{n-1-j}=l$.
Let $k\in\{0,\dots,[(n-1)/2]+3\}$ and let $t\in(h_K-\de^2,h_K)$. Since $r(t)$ is differentiable $m$ times and $m\geq[(n-1)/2]+3$, the derivative of order $k$ of this term with espect to $t$  exists and is a linear combination  of terms of the form
\begin{equation*}
\frac{\pa^{k-p}(h_K-t)^{l/2}}{\pa t^{k-p}}\int_{\Sndue} c_{i_1}\dots c_{i_{n-1-j}} \frac{\pa^p \left(r\left(\sqrt{h_K-t}\right)\right)^j}{\pa t^p} \,d\te,
\end{equation*}
with $0\leq p\leq k$. 
All this, \eqref{stima_resto} and the continuity of $(\pa^m/\pa t^m) R(t)$ imply that the derivative of order $k$ of $I_2(t)$ is a linear combination  of terms which are continuous and whose asymptotic behavior as $t<h_K$ tends to $h_K $ is $o\left((h_K-t)^{(mj+l)/2-k}\right)$. Note that this asymptotic behavior is locally uniform with respect to $u$.
Since
\[
m\geq \begin{cases}6 &\text{when $n$ is even,}\\7 &\text{when $n$ is odd}\end{cases}
\]
the exponent $(mj+l)/2-k$ is nonnegative for $k$, $j$ and $l$ in the ranges described above (because  $mj+l\geq m+n-2$ and $k\leq [(n-1)/2]+3$).
This concludes the proof of Assertion~\eqref{ass:lem_kobayashi_cm_II}.

\emph{Assertion~\eqref{ass:lem_kobayashi_cm_III}.}
The coefficients $b_1(u)$ and $b_2(u)$ in \eqref{Iunomenoprimitermini} are integrals over $\Sndue$ of polynomials in the coefficients $c_1,\dots, c_5$ of \eqref{taylor_exp_R}.
These coefficients are, up to constants, the derivatives with respect to $t$, up to order five, of $f_1$ at $t=0$, or equivalently,  the derivatives with respect to $t$, of order up to seven, of $f$ at $t=0$.
The regularity of $f$ implies that  $b_1(u)$ and $b_2(u)$ depends continuously on $u$. The same is true for $b_0(u)$, due to its explicit representation and the regularity of $\gau_K$.
The assertion regarding the boundedness of the $[(n-1)/2]+3$  derivative with respect to $t$ of the expression in~\eqref{taylor_radon} in a left neighborhood of $h_K(u)$ is a consequence of what we have proved above regarding the
$[(n-1)/2]+3$ derivative of the expression in~\eqref{Iunomenoprimitermini} and of $I_2(t)$.
\end{proof}

The next lemma is used in proving the analyticity of the maps $F_{m,K}$ appearing in Theorem~\ref{teo_kobayashi}. This property  is a consequence of the fact that $\fou{1_K}$ is holomorphic, on the analytic Implicit Function Theorem and on the fact that if $\ze u$ is a zero of $\fou{1_K}$ and if $\re\ze$ is sufficiently large, then
\begin{equation*}
\frac{\pa}{\pa\ze}\fou{1_K}(\ze u)\neq 0.
\end{equation*} 
The next lemma is relevant for the asymptotic behavior of this derivative, which  coincides with the Fourier-Laplace transform with respect to $t$ of $\ii t S_K(u,t)$.

\begin{lemma}\label{lem_kobayashi_cm_itS}
Let $\ds$ be as in Theorem~\ref{teo_cov_smooth}, let $K\subset\R^n$ be a convex body of class $C^\ds_+$ with $o\in\inte K$ 
and let $V$ and $\phi$ be as in Lemma~\ref{lem_kobayashi_cm}. 
Then for every $u\in\Sn$ and $j=0,1,2$ there exists  
$\widetilde{a}_j(u)$ and $\widetilde{b}_j(u)$ in $\R$ such that $\widetilde{a}_0\neq0$, $\widetilde{b}_0\neq0$,
\begin{multline}\label{taylor_radon_itS}
 \ii t S_K(u,t)-\sum_{j=0}^2\widetilde{a}_j(u)\Big(t+h_K(-u)\Big)^{\frac{n-1}2+j}_+
\phi\big(u,t+h_K(-u)\big)+\\-\sum_{j=0}^2\widetilde{b}_j(u)\Big(t-h_K(u)\Big)^{\frac{n-1}2+j}_-
\phi\big(u,t-h_K(u)\big),
\end{multline}
as a function of $t$, belongs to $C^{\left[(n-1)/2\right]+2}(\R)$, and its derivative of order $\left[(n-1)/2\right]+3$ exists in $(-h_K(-u), h_K(u))$.
Moreover  $|\widetilde{a}_j(u)|$ and $|\widetilde{b}_j(u)|$, for $j=0,1,2$, and 
\[
 \sup_{-h_K(-u)< t< h_K(u)}\left|\frac{\pa^{\frac{n-1}{2}+3}\big(\text{expression in \eqref{taylor_radon_itS}}\big)}{\pa t^{\frac{n-1}{2}+3}}\right|
\]
are  bounded from above uniformly with respect to $u$ in $\Sn$.
\end{lemma}
\begin{proof}
Let $a_j(u)$ and $b_j(u)$, $j=0,1,2$, be the coefficients defined in the statement of Lemma~\ref{lem_kobayashi_cm} and let $H(u,t)$ denote the expression in~\eqref{taylor_radon}. By multiplying $H$ by $\ii t$ we can rewrite it as
\begin{multline*}\label{taylor_radon_itS}
 \ii t S_K(u,t)
 -\sum_{j=0}^2\widetilde{a}_j(u)\Big(t+h_K(-u)\Big)^{\frac{n-1}2+j}_+
\phi\big(u,t+h_K(-u)\big)+\\
  -\sum_{j=0}^2\widetilde{b}_j(u)\Big(t-h_K(u)\Big)^{\frac{n-1}2+j}_-
\phi\big(u,t-h_K(u)\big)=\widetilde{H}(u,t)
\end{multline*}
where, for $j=1,2$, 
\begin{align*}
\widetilde{a}_0(u)&=-\ii h_K(-u) a_0(u), & \widetilde{b}_0(u)&=\ii h_K(u) b_0(u),  \\
\widetilde{a}_j(u)&=\ii\left(a_{j-1}(u)-a_j(u)h_K(-u)\right), & \widetilde{b}_j(u)&=\ii\left(b_j(u)h_K(u)-b_{j-1}(u)\right)
\end{align*}
and 
\begin{multline*}
 \widetilde{H}(u,t)=
 \ii t H(u,t)
 +\ii a_2(u)\Big(t+h_K(-u)\Big)^{\frac{n-1}2+3}_+\phi\big(u,t+h_K(-u)\big)+\\
 -\ii b_2(u)\Big(t-h_K(u)\Big)^{\frac{n-1}2+3}_-\phi\big(u,t-h_K(u)\big).
\end{multline*}
The assumption $o\in\inte{K}$ imply $-h_K(-u)<0<h_K(u)$, and this imply $\widetilde{a}_0(u)\neq0$ and $\widetilde{b}_0(u)\neq0$. The results about $H$, $a_j$ and $b_j$ proved in  Lemma~\ref{lem_kobayashi_cm} give the other conclusions of this lemma. 
\end{proof}

\begin{proof}[Proof of Theorem~\ref{teo_kobayashi_cm}]Let us write $\fou{1_K}(\ze u)$ as in \eqref{radon_plus_fourier_oned} and let us apply \cite[Corollary 2.20]{Kob2} to the Fourier-Laplace transform of $S_K(u,t)$ with respect to $t$
\begin{equation}\label{FT_radon}
 \fou{S_K(u,t)}(\ze):=\int_{-\infty}^{\infty}S_K(u,t)e^{\ii t\ze}\,dt.
\end{equation}
In the terminology of \cite{Kob2} (see in particular \cite[p. 20]{Kob2}) our Lemma~\ref{lem_kobayashi_cm} proves that $S_K(u,t)$, as a function of $t$, belongs to $\mathscr{C}^2((n-1)/2)$ with $\alpha=-h_K(-u)$, $\beta=h_K(u)$, $A(f)=\width_K(u)$, $a_j(f)=a_j(u)$ and $b_j(f)=b_j(u)$ for $j=0,1,2$. Let $\|S_K(u,t)\|_{\mathscr{C}^2((n-1)/2)}$ denote 
\begin{multline*}
 \sum_{j=0}^2\width_K(u)^{\frac{n-1}{2}+j}\left(|a_j(u)|+|b_j(u)|\right)+\\
 +\width_K(u)^{\frac{n-1}{2}+3}\sup_{-h_K(-u)< t< h_K(u)}\left|\frac{\pa^{\frac{n-1}{2}+3}\big(\text{expression in \eqref{taylor_radon}}\big)}{\pa t^{\frac{n-1}{2}+3}}\right|.
\end{multline*}
By Lemma~\ref{lem_kobayashi_cm}, $\sup_{u\in\Sn}\|S_K(u,t)\|_{\mathscr{C}^2((n-1)/2)}$ is finite.

\cite[Corollary 2.20]{Kob2} applies and proves that for each $u\in\Sn$ there exist a positive integer $m(K,u)$, a positive number $d(K,u)$ and a finite set $C(K,u)\subset\C$ such that the zero set of $\fou{S_K(u,t)}$ consists of  $C(K,u)$ and, for each $m\geq m(K,u)$, of one simple zero in each  of the two balls (in $\C$)
\begin{equation}\label{balls_containing_zero}
 B\left(
 \ga\frac{\pi(4m+n-1)}{2\width_K(u)}+\ii\ \frac{\ln \gau_K(-u)-\ln \gau_K(u)}{2\width_K(u)},\frac{d(K,u)}{m}
 \right),
 \quad\text{$\ga=1,-1$.}
\end{equation}
Moreover $m(K,u)$, $d(K,u)$ and the radius of a ball centered at $o$ and containing $C(K,u)$ are bounded from above uniformly with respect to $u\in\Sn$ in terms of $\sup_{u\in\Sn}\|S_K(u,t)\|_{\mathscr{C}^2((n-1)/2)}$, $\inf_{u\in\Sn}\width_K(u)$, $\sup_{u\in\Sn}\width_K(u)$, $\inf_{u\in\Sn}\gau_K(u)$ and $\sup_{u\in\Sn}\gau_K(u)$.

Let $m(K)=\sup_{u\in\Sn}m(K,u)$ and, for each $m\geq m(K)$, let $F_{m,K}(u)$ be the zero of $\fou{S_K(u,t)}$ contained in the ball in~\eqref{balls_containing_zero} corresponding to $\ga=1$. (The one corresponding to $\ga=-1$ coincides with $-F_{m,K}(-u)$.)
Due to \eqref{radon_plus_fourier_oned} the intersection of the zero set of $\fou{1_K}$ with the ray $\{z=\ze u:\ze\in\C\}$ consists of a bounded set and of $\cup_{m\geq m(K)}\{F_{m.K}(u)u, -F_{m,K}(-u)u\}$.

To complete the proof it remains to prove that the map $F_{m,K}:\Sn\to\C$ is analytic.
In view of the analyticity of $\fou{1_K}$ and of the analytic Implicit Function Theorem it suffices to prove that if $\ze u$ is a zero of $\fou{1_K}$ and if $\re\ze$ is sufficiently large, then
\begin{equation}\label{derivative_different_from_zero}
\frac{\pa}{\pa\ze}\fou{1_K}(\ze u)\neq 0.
\end{equation}

For a $C^\infty_+$ convex body $K$ this last formula is proved in \cite[Lemma 2.4.25]{Kob1}. The only point of the proof of this lemma where the regularity of $K$  enters is in the asymptotic expansion of $({\pa}/{\pa\ze})\fou{1_K}(\ze u)$ given by~\cite[Formula~(2.4.27)]{Kob1}. If we prove this formula in the case of a $C^{\ds}_+$ set then all the rest of the proof goes unchanged.

To prove this formula for $K\in\C^\ds_+$ one argues as follows. We may assume $o\in\inte K$, because a translation of $K$ does not change $\cZ(K)$ (we have $\fou{1_{K+y}}(\zeta)=e^{y\zeta}\fou{1_K}(\zeta)$ for $y\in\R^n$). The function $({\pa}/{\pa\ze})\fou{1_K}(\ze u)$ coincides with the Fourier-Laplace transform with respect to $t$ of $\ii t S_K(u,t)$. Lemma~\ref{lem_kobayashi_cm_itS} proves that $\ii t S_K(u,t)$, as a function of $t$, belongs to $\mathscr{C}^2((n-1)/2)$.
\cite[Lemma 2.13]{Kob2} (with $f=\ii t S_K(u,t)$, $\la=(n-1)/2$, $\al(f)=-h_K(-u)$, $\be(f)=h_K(u)$, $a_0(f)=\widetilde{a}_0(u)$, $b_0(f)=\widetilde{b}_0(u)$, $A(f)=\width_K(u)$ and $p(\la)=\Gamma\left((n+1)/2\right)e^{\ii\pi(n+1)/4}$) applies  and yields \cite[Formula (2.4.27)]{Kob1} for any $\ze$ such that $\ze u$ is a zero of $\fou{1_K}$ and $\re\ze$ is sufficiently large.
\end{proof}

\section{Covariogram Problem for $C^\ds_+$ regular bodies}\label{sec_cov}
Kobayashi result enters the proof of Theorem~\ref{teo_cov_smooth} only through the next proposition. The key  point in the proof of this proposition is the fact that the maps $F_{m,K}$, introduced in the statement of Theorem~\ref{teo_kobayashi}, are analytic.
\begin{proposition}\label{prop_ratio_radii_curvature}
Let $H$, $K$ be convex bodies of class $C^\ds_+$ with $g_H=g_K$. Then
\begin{align*}\label{equal_ratio_radii_curvature}
\text{either}\quad\frac{\gau_H(-u)}{\gau_H(u)}=& \frac{\gau_K(-u)}{\gau_K(u)}\quad\text{for each } u\in\Sn\\
\text{or}\quad\frac{\gau_H(u)}{\gau_H(-u)}=& \frac{\gau_K(-u)}{\gau_K(u)}\quad\text{for each } u\in\Sn.
\end{align*}
\end{proposition}

\begin{proof}
The identity $g_H=g_K$, \eqref{width_of_support} and \eqref{convoluzione_in_cn} imply
\begin{equation}\label{equal_width}
\width_H=\width_K
\end{equation}
and, for $\zeta\in\C^n$,
\begin{equation*}
\fou{1_H}(\zeta)\,
\conj{\fou{1_{H}}\left(\cz\right)}=
\fou{1_K}(\zeta)\,
\conj{\fou{1_{K}}\left(\cz\right)}.
\end{equation*}
Thus we have
\begin{equation}\label{identita_zeri_covario}
\cZ(H)\bigcup\conj{\cZ(H)}=\cZ(K)\bigcup\conj{\cZ(K)}.
\end{equation}
Let us use the notations introduced in the statement of Theorem~\ref{teo_kobayashi}. Let us choose $m_0>m(K),m(H)$ such that $\cZ_{m,K}\cap C(H)=\emptyset$ for each $m\geq m_0$. 
Theorem \ref{teo_kobayashi_cm} and \eqref{identita_zeri_covario} imply that for each $m\geq m_0$ and for each $u\in\Sn$ we have
\[
 F_{m,K}(u)u\in \bigcup_{l=m(H)}^\infty \left(\cZ_l(H)\bigcup\conj{\cZ_l(H)}\right).
\]
The representation of $\cZ_l(H)$ provided by \eqref{rappresentazione_Zm_con_Fm}  implies that there exists $l=l(m,u)$ such that either $F_{m,K}(u)=F_{l,H}(u)$ or $F_{m,K}(u)=\conj{F_{l,H}(u)}$.
In both cases the representation of the real parts  of $F_{m,K}$ and $F_{l,H}$ given in \eqref{rappresentazione_mappa_analitica}, together with \eqref{equal_width}, implies that there exists $m_1\geq m_0$ such that   $l=m$ for each $m\geq m_1$ and $u\in\Sn$. Summarizing, for each $u\in\Sn$ and $m\geq m_1$  either we have
\begin{gather}
 F_{m,K}(u)=F_{m,H}(u)\label{alternativa1}
\intertext{or we have}
F_{m,K}(u)=\conj{F_{m,H}(u)}.\label{alternativa2}
\end{gather}
A priori the choice may vary from $u$ to $u$. We may assume that $K$ is not centrally symmetric because otherwise $K$ is a reflection or translation of $H$ (it is an easy consequence of the Brunn-Minkowski inequality as explained, for instance,  in~\cite[p. 204]{Bianchi-2005}) and the claim follows. We may thus assume that $\gau_K$ is not an even function. Formula~\eqref{rappresentazione_mappa_analitica} implies that there exists $m_2\geq m_1$ and a relatively open connected subset $U$ of $\Sn$ such that
\begin{equation}\label{zero_set_K_with_im_positive}
\im F_{m,K}(u)>0,
\end{equation}
for each $m\ge m_2$ and $u\in U$. The alternatives  \eqref{alternativa1} and \eqref{alternativa2} imply that $\im F_{m,H}(u)\neq0$ when $u\in U$.
Formula \eqref{riflessione_in_Cn} implies that passing from $H$ to $-H$ corresponds to conjugating $F_{m,H}$.  Thus, possibly after a reflection of $H$, there is $m_3\geq m_2$ and  a relatively open set $V\subset U$ such that
\begin{equation}\label{zero_set_H_with_im_positive}
\im F_{m,H}(u)>0,
\end{equation}
for each $m\geq m_3$ and $u\in V$. Formulas~\eqref{alternativa1}, \eqref{alternativa2}, \eqref{zero_set_K_with_im_positive} and~\eqref{zero_set_H_with_im_positive}  imply
\begin{equation*}
 F_{m,H}(u)=F_{m,K}(u)
\end{equation*}
for each $m\geq m_3$ and $u\in V$. Since $F_{m,K}$ and $F_{m,H}$ are analytic maps from $\Sn$ to $\C$, their coincidence in $V$ implies their coincidence on the whole $\Sn$, i.e.
\[
 F_{m,H}=F_{m,K}.
\]
This and  \eqref{rappresentazione_mappa_analitica} conclude the proof.
\end{proof}

\begin{proof}[Proof of Theorem~\ref{teo_cov_smooth}]
 Propositions~\ref{teo_information_c2+_bodies}-\eqref{nonordered_curvatures} and~\eqref{prop_ratio_radii_curvature} imply that, possibly after a reflection of $H$, we have
\[
\gau_H(u)=\gau_K(u)
\]
for each $u\in\Sn$. The uniqueness part in Minkowski's Theorem \cite[Th. 7.2.1]{Sc} implies that $H$ and $K$ coincide, up to translations.
\end{proof}
\begin{remark}
Theorem~\ref{teo_cov_smooth} only proves that the covariogram determines a $\C^\ds_+$ body among $\C^\ds_+$ bodies. We are not able to prove that the determination holds among  all convex bodies.
\end{remark}

\section{Cross covariogram Problem for $C^\ds_+$ regular bodies}\label{sec_cross_cov}
Let $H$ and $K$ be convex bodies in $\R^n$. The translation of $H$ and $K$ by the same vector, and the substitution of $H$ with $-K$ and of $K$ with $-H$, leave $g_{H,K}$ unchanged.  We call $(H,K)$ and $(H',K')$ \emph{trivial associates} when one pair is obtained by the other one via a combination of the two operations above.
\smallskip

\textbf{Cross covariogram Problem.}
\emph{Does $g_{H,K}$ determine the pair $(H,K)$ of convex bodies among all pairs of convex bodies, up to trivial associates?}
\smallskip

Bianchi~\cite{B4} gives a complete answer to this problem when $H$ and $K$ are convex polygons. In order to explain this result let us introduce some families of sets.
\begin{example}\label{parallelograms}
Let $\al$, $\be$, $\ga$, $\de$, $\al'$, $\be'$, $\ga'$ and $\de'$ be positive real numbers, $m\in\R$, $y,y'\in\R^2$, $I_1=[(-1,0),(1,0)]$, $I_2=1/\sqrt{2}\ [(-1,-1),(1,1)]$, $I_3=[(0,-1),(0,1)]$, $I_4=1/\sqrt{2}\ [(1,-1),(-1,1)]$ and $I_5=(1/\sqrt{1+m^2})\,[(-m,-1),(m,1)]$. Assume either  $m=0$, $\al'\neq \ga'$ and $\be'\neq\de'$ or else $m\neq 0$ and $\al'\neq \ga'$. We define four pairs of parallelograms as follows (see Figure~\ref{fig_four_parall}):
\begin{align*}
\cH_1&=\al I_1+\be I_2,\quad& \cK_1&=\ga I_3+\de I_4+y;\\
\cH_2&=\al I_1+\de I_4,\quad&  \cK_2&=\be I_2+\ga I_3+y;\\
\cH_3&=\al' I_1+\be' I_3,\quad& \cK_3&=\ga' I_1+\de' I_5+y';\\
\cH_4&=\ga' I_1+\be' I_3,\quad& \cK_4&=\al' I_1+\de' I_5+y'.
\end{align*}

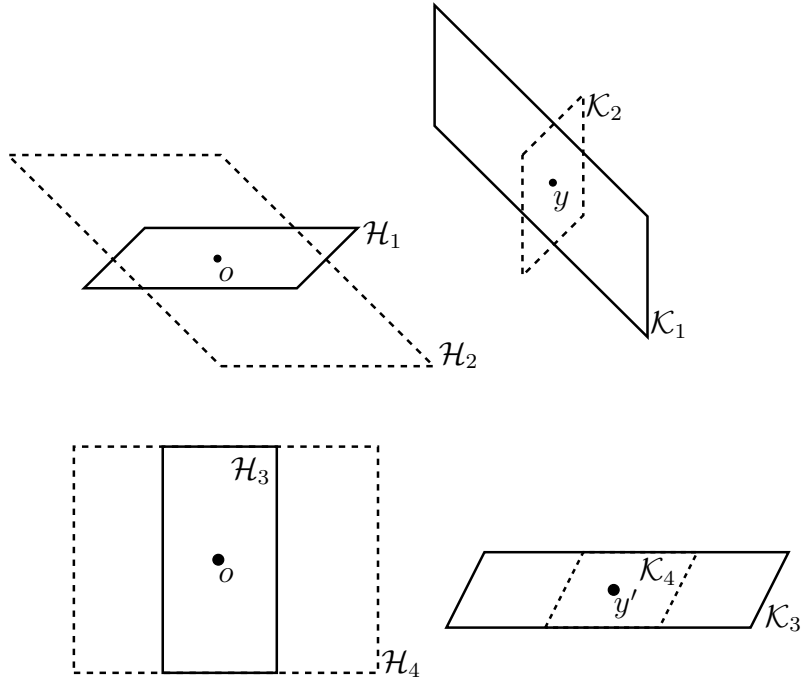
\begin{figure}
\begin{center}
\input{quattro_paral_ver2.pspdftex}
\end{center}
\caption{We have $g_{\cH_1,\cK_1}=g_{\cH_{2},\cK_{2}}$ and $g_{\cH_3,\cK_3}=g_{\cH_{4},\cK_{4}}$. Moreover, up to affine transformations, these are the only pairs of planar convex polygons with equal cross covariogram.}
\label{fig_four_parall}
\end{figure}

\cite{B4} proves that for $i=1,3$, we have $g_{\cH_i,\cK_i}=g_{\cH_{i+1},\cK_{i+1}}$ but $(\cH_i,\cK_i)$ is not a trivial associate of $(\cH_{i+1},\cK_{i+1})$. It also proves that, in the class of convex polygons and up to an affine transformation, the previous counterexamples are the only ones.
\end{example}

\begin{theorem}[Bianchi~\cite{B4}]\label{teo_cov_congiunto_poligoni}
Let $H$ and $K$ be convex polygons and $H'$ and $K'$ be planar  convex bodies
with $g_{H,K}=g_{H',K'}$. Assume that there exist no affine transformation $\cT$ and no different  indices $i,j$, with either $i,j\in\{1,2\}$ or  $i,j\in\{3,4\}$, such that $(\cT H,\cT K)$ and $(\cT H',\cT K')$ are trivial associates of $(\cH_i,\cK_i)$ and of $(\cH_j,\cK_j)$, respectively.
Then  $(H,K)$ is a trivial associate of $(H',K')$.
\end{theorem}

In this paper we are able to prove that no counterexample exists among pairs of sufficiently regular planar convex bodies.

\begin{theorem}\label{teo_cross_cov_smooth}Let $H, K, H'$ and $K'$ be  planar convex bodies of class $C^8_+$. Then $g_{H,K}=g_{H',K'}$ implies that $(H,K)$ and $(H',K')$ are trivial associates. \end{theorem}
\begin{proof}
Formula~\eqref{convoluzione} implies
\[
 \fou{1_H}(\zeta)\,{\fou{1_{-K}}(\zeta)}=\fou{1_{H'}}(\zeta)\,{\fou{1_{-K'}}(\zeta)}
\]
and, as a consequence,
\[
\cZ(H)\bigcup{\cZ(-K)}=\cZ(H')\bigcup{\cZ(-K')}, 
\]
where, for a convex body $L\subset\R^n$, $\cZ(L)=\{\ze\in\C^n : \fou{1_L}(\ze)=0\}$.
This identity implies, by Theorem~\ref{teo_kobayashi_cm}, the existence of positive integers $m_i$, $i=1,2,3,4$,  such that
for each $u\in\Sdue$
\begin{multline}\label{union_of_zero_curves}
\left\{F_{m,H}(u):m\geq m_1\right\}\bigcup\left\{F_{m,-K}(u):m\geq m_2\right\}=\\
=\left\{F_{m,H'}(u):m\geq m_3\right\}\bigcup\left\{F_{m,-K'}(u):m\geq m_4\right\}.
\end{multline}

We first show that for each $u\in\Sdue$ we have
\begin{equation}\label{equal_set_widths}
 \{\width_H(u),\width_K(u)\}=\{\width_{H'}(u),\width_{K'}(u)\}.
\end{equation}
Formula~\eqref{width_of_support} implies
\begin{equation}\label{equal_sums_widths}
 \width_H+\width_{K}=\width_{H'}+\width_{K'}.
\end{equation}
Let $u\in\Sn$. If one of the elements of $\{\width_H(u),\width_K(u)\}$ belongs to $\{\width_{H'}(u),\width_{K'}(u)\}$, then~\eqref{equal_sums_widths} implies~\eqref{equal_set_widths}. If $\width_H(u)=\width_K(u)$ and $\width_{H'}(u)=\width_{K'}(u)$, then again~\eqref{equal_sums_widths} implies~\eqref{equal_set_widths}. Thus we may assume that one of the four numbers $\width_H(u)$, $\width_K(u)$, $\width_{H'}(u)$ and $\width_{K'}(u)$ is strictly larger than the other ones. Let us assume
\begin{equation}\label{strictly_larger_widths}
 \width_H(u)>\max\left\{\width_K(u),\width_{H'}(u), \width_{K'}(u)\right\}.
\end{equation}
(The other cases can be treated similarly.) By~\eqref{union_of_zero_curves} for each $m\geq m_1$ and $i=0,1,2$ there exists $l_i=l_i(m,u)$ such that
\begin{align}
 F_{m+i,H}(u)&=F_{l_i,H'}(u)\label{cc_alternative1}\\
 \text{or }\quad F_{m+i,H}(u)&=F_{l_i,-K'}(u). \label{cc_alternative2}
\end{align}
If \eqref{cc_alternative1} (if \eqref{cc_alternative2}) holds for a particular value of $i$ we say that \eqref{cc_alternative1}$_i$ (\eqref{cc_alternative2}$_i$, respectively) holds. 
At least one between \eqref{cc_alternative1}$_0$ and \eqref{cc_alternative2}$_0$ holds for infinitely many values of $m$, and let us assume that this happen for~\eqref{cc_alternative1}$_0$ (the other case can be treated similarly).
Note that~\eqref{cc_alternative1}$_0$ and~\eqref{cc_alternative1}$_1$ do not hold together when $m$ is sufficiently large. Indeed if they do we have $F_{m+1,H}(u)-F_{m,H}(u)=F_{l_1,H'}(u)-F_{l_0,H'}(u)$. On the other hand we have 
\begin{gather*}
  \re\left(F_{m+1,H}(u)-F_{m,H}(u)\right)=\frac{2\pi}{\width_H(u)}+\bigO\left(\frac1{m}\right),\\
 \re\left(F_{l_1,H'}(u)-F_{l_0,H'}(u)\right)=\frac{2\pi(l_1-l_0)}{\width_{H'}(u)}+\bigO\left(\frac1{m}\right)
\end{gather*}
(the term $\bigO(1/m)$ in the first line may differ from that in the second line) and the right-hand side of the first equation is strictly less than the right-hand side of the second equation when $m$ is sufficiently large, due to~\eqref{strictly_larger_widths} and $l_0<l_1$.  A similar argument proves that~\eqref{cc_alternative2}$_1$ and~\eqref{cc_alternative2}$_2$ do not hold together when $m$ is sufficiently large.
Thus~\eqref{cc_alternative1}$_0$ and~\eqref{cc_alternative1}$_2$ hold for all $m$ in an infinite set $I$. When $m\in I$ we have
\begin{gather*}
  \re\left(F_{m+2,H}(u)-F_{m,H}(u)\right)=\frac{4\pi}{\width_H(u)}+\bigO\left(\frac1{m}\right),\\
 \re\left(F_{l_2,H'}(u)-F_{l_0,H'}(u)\right)=\frac{2\pi(l_2-l_0)}{\width_{H'}(u)}+\bigO\left(\frac1{m}\right).
\end{gather*}
Arguing as above proves that $l_2-l_0=1$ when $m\in I$ and $m$ is large enough. This implies $\width_H(u)=2\width_{H'}(u)$. Thus Theorem~\ref{teo_kobayashi_cm} implies that when \eqref{cc_alternative1}$_0$ holds we have 
 \begin{equation*}
  \frac{\pi (4m+1)}{2\width_H(u)}=\frac{\pi (4l_0+1)}{\width_H(u)}+\bigO\left(\frac1{m}\right).
 \end{equation*}
This implies
\[
m-2l_0=1/4+\bigO(1/m), 
\]
This equality does not hold when $m$ is large, because $m-2l_0\in\mathbb{Z}$ while $1/4+\bigO(1/m)\notin\mathbb{Z}$. This contradiction concludes the proof of~\eqref{equal_set_widths}.

Assume that there exists a relatively open subset $U$ in $\Sdue$ such that
\begin{equation}\label{locally_diff_widths}
 \width_H(u)\neq\width_K(u)\quad\text{for each $u\in U$.}
\end{equation}
Up to restricting $U$ we may assume that
\begin{align}
\text{either $\width_H(u)=\width_{H'}(u)$ and $\width_K(u)=\width_{K'}(u)$ for each $u\in U$}\label{cc_alternative3}\\
 \text{or $\width_H(u)=\width_{K'}(u)$ and $\width_K(u)=\width_{H'}(u)$ for each $u\in U$}\label{cc_alternative4}.
\end{align}
Let us assume that \eqref{cc_alternative3} holds. (The other case can be treated similarly.)
Formulas~\eqref{rappresentazione_mappa_analitica}, \eqref{union_of_zero_curves}, \eqref{locally_diff_widths} and \eqref{cc_alternative3} imply that for each integer $m\geq m_1$ and for each $u\in U$ we have
\begin{equation}\label{id1}
 F_{m,H}(u)=F_{m,H'}(u)\quad\text{ and }\quad F_{m,-K}(u)=F_{m,-K'}(u)
\end{equation}
Since the four maps appearing in~\eqref{id1} are analytic maps from $\Sdue$ to $\C^2$, we have $F_{m,H}(u)=F_{m,H'}(u)$ and $F_{m,-K}(u)=F_{m,-K'}(u)$ for each $u\in\Sdue$. The equalities of the real parts imply
\begin{equation}\label{equal_widths_global}
 \width_H=\width_{H'}\quad\text{ and }\quad\width_K=\width_{K'}.
\end{equation}
The equalities of the imaginary parts imply
\begin{equation}\label{equal_ratio_curv_2}
\frac{\gau_{H}(-u)}{\gau_{H}(u)}=\frac{\gau_{H'}(-u)}{\gau_{H'}(u)}\text{ and }
\frac{\gau_{-K}(-u)}{\gau_{-K}(u)}=\frac{\gau_{-K'}(-u)}{\gau_{-K'}(u)}
\end{equation}
for each $u\in\Sdue$. 
By Proposition~\ref{teo_information_c2+_bodies}-\eqref{prop_sum_reverse_weingarten} the identities \eqref{equal_widths_global} imply
\begin{equation*}
\begin{aligned}
\frac1{\gau_{H}(u)}+\frac1{\gau_{H}(-u)}&=\frac1{\gau_{H'}(u)}+\frac1{\gau_{H'}(-u)}, \\
\frac1{\gau_{-K}(u)}+\frac1{\gau_{-K}(-u)}&=\frac1{\gau_{-K'}(u)}+\frac1{\gau_{-K'}(-u)}
\end{aligned}
\end{equation*}
for each $u\in\Sdue$. All these conditions imply $\gau_H=\gau_{H'}$ and $\gau_{-K}=\gau_{-K'}$.
The uniqueness part in Minkowski's Theorem \cite[Th. 7.2.1]{Sc} implies   $H=H'+x_1$ and $K=K'+x_2$, for suitable $x_1,x_2\in\R^2$.
The identity $H+(-K)=\supp\, g_{H,K}=\supp\, g_{H',K'}=H'+(-K')$ implies $x_1=x_2$. This concludes the proof under Assumption~\eqref{locally_diff_widths}.

If~\eqref{locally_diff_widths} does not hold, then~\eqref{equal_sums_widths} implies
\begin{equation}\label{all_equal_widths_global}
\width_H=\width_{H'}=\width_K=\width_{K'}.
\end{equation}
We again distinguish two cases according to whether 
\begin{equation}\label{equal_ratio_curv}
\frac{\gau_{H}(-u)}{\gau_{H}(u)}=\frac{\gau_{-K}(-u)}{\gau_{-K}(u)}
\end{equation}
holds for each $u\in \Sdue$ or not.
If~\eqref{equal_ratio_curv} holds for each $u\in \Sdue$ then, arguing as we have done above we conclude that $H=-K$. This implies 
\[
 F_{m,H}=F_{m,-K}
\]
for each $m$. This, \eqref{rappresentazione_mappa_analitica}, \eqref{union_of_zero_curves} and  \eqref{all_equal_widths_global} imply $F_{m,H}=F_{m,H'}=F_{m,-K'}$ for each $m$ sufficiently large. This implies $\gau_{H}(-u)/\gau_{H}(u)=\gau_{H'}(-u)/\gau_{H'}(u)=\gau_{-K'}(-u)/\gau_{-K'}(u)$ for each $u\in \Sdue$ and $H=H'=-K'$. The proof is concluded in this case too.

It remains to consider the possibility that there exists a relatively open subset $U$ of $\Sdue$ such that~\eqref{equal_ratio_curv} is false  for each $u\in U$. This and~\eqref{all_equal_widths_global} imply that when $u\in U$ the real parts of $F_{m,H}(u)$ and $F_{m,-K}(u)$ coincide  but their imaginary parts differ.
Formula~\eqref{union_of_zero_curves} and the analyticity of $F_{m,H}$, $F_{m,-K}$, $F_{m,H'}$ and $F_{m,-K'}$ imply that we have
\begin{gather}
\text{$F_{m,H}=F_{m,H'}$ and $F_{m,-K}=F_{m,-K'}$ for infinitely many $m$}\label{cc_alternative5}\\
\text{or $F_{m,H}=F_{m,-K'}$ and $F_{m,-K}=F_{m,H'}$ for infinitely many $m$.}\label{cc_alternative6}
\end{gather}
If~\eqref{cc_alternative5} holds then we have~\eqref{equal_ratio_curv_2} and we conclude as before. When~\eqref{cc_alternative6} holds the proof is concluded by similar arguments.
\end{proof}

\section{The Covariogram Problem and irreducibility of $\fou{1_K}$}
\label{sec_phase_retr}

We say that an entire function $g$ is \emph{irreducible} if $g$ cannot be written as the product of two entire functions $g_1$, $g_2$ with $g_1\neq\al g$, for each $\al\in\C$,  and both $\{\ze\in\C^n : g_1(\ze)=0\}$ and $\{\ze\in\C^n : g_2(\ze)=0\}$ nonempty. Let $f\in L^2(\R^n)$ have compact support.
Sanz and Huang~\cite{Sanz-Huang-1984} proves that if $\fou{f}$ is irreducible then $f$ is determined, up to trivial associates, by the knowledge of $|\fou{f}(x)|$ for all $x\in\R^n$.
Barakat and Newsam~\cite{Barakat-Newsam-1984} and Stefanescu~\cite{Stefanescu-1985} prove that if $f_1$ and $f_2$ belong to $L^2(\R^2)$, have compact support, are not trivial associates and $|\fou{f_1}(x)|=|\fou{f_2}(x)|$ for all $x\in\R^2$, then there exist two entire functions $g_1$ and $g_2$ such that $\{\ze\in\C^2 : g_1(\ze)=0\}$ and $\{\ze\in\C^2 : g_2(\ze)=0\}$ are both nonempty and
\begin{equation}\label{factor_barakat}
\fou{f_1}(\ze)=g_1(\ze)g_2(\ze)\quad\text{ and }\quad \fou{f_2}(\ze)=e^{\ii(c+ \left<d,\ze\right>)}g_1(\ze)\conj{g_2\left(\conj{\ze}\right)},
\end{equation}
for a suitable $c\in\R$ and $d\in\R^2$. Stefanescu~\cite{Stefanescu-personal} believes that a similar result holds true in any dimension $n\geq2$. It is not known whether the property that $\fou{f}$ is not irreducible implies that $f$ is not determined by $|\fou{f}|$.

What is the significance of these results for the Covariogram Problem? Assume that $n=n_1+n_2$, with $n_1$, $n_2$ positive integers, and that the convex body $K\subset\R^n=\R^{n_1}\times\R^{n_2}$ can be written as
\begin{equation}\label{decomp_direct_sum}
 K=K_1+K_2
\end{equation}
with $K_1\subset \R^{n_1}$ and $K_2\subset\R^{n_2}$ convex bodies which are not centrally symmetric.
Then  $K'=K_1+(-K_2)$ is not  a translation or reflection of $K$ and $g_K=g_{K'}$ (see Bianchi~\cite{Bianchi-2009-polytopes}). All known examples of nondetermination for the Covariogram Problem arise, up to a linear transformation, by a decomposition as in~\eqref{decomp_direct_sum}. This decomposition generates a factorization of $\fou{1_K}$ as in \eqref{factor_barakat}. Indeed
\[
 1_{K}=\delta_{K_1}\ast \delta_{K_2}\quad\text{and}\quad 1_{K'}=\delta_{K_1}\ast \delta_{-{K_2}},
\]
where $\delta_{K_1}$ and $\de_{K_2}$ are the distributions defined for $\phi\in C^\infty_0(\R^{n})$ by
\[
 \delta_{K_1}(\phi)=\int_{K_1} \phi(x,0)\, dx,\quad \delta_{K_2}(\phi)=\int_{K_2} \phi(0,y)\, dy
\]
(here $x\in \R^{n_1}$, $y\in\R^{n_2}$ and $dx$ and $dy$ denote, respectively, Lebesgue measure in $\R^{n_1}$ and in $\R^{n_2}$) and $\de_{-{K_2}}$ is defined similarly. By the Paley-Wiener Theorem $\fou{\de_{K_1}}$, $\fou{\de_{K_2}}$ and $\fou{\de_{-{K_2}}}$ are entire functions in $\C^{n}$ of exponential type. Clearly $\fou{\delta_{-{K_2}}}(\ze)=\conj{\fou{\delta_{K_2}}\left(\conj{\ze}\right)}$ and we have
\[
\fou{1_{K}}=\fou{\delta_{K_1}}\fou{\delta_{K_2}}\quad\text{and}\quad \fou{1_{K'}}(\ze)=\fou{\delta_{K_1}}(\ze) \conj{\fou{\delta_{K_2}}\left(\conj{\ze}\right)},
\]
as in~\eqref{factor_barakat}.

In view of these results it would be interesting to \emph{find explicit geometric conditions on a convex body $K$ which grants that $\fou{1_K}$ is irreducible}. Regarding the difficulty in answering to this question, consider the following subproblem.
\smallskip

\emph{Understand for which convex bodies $K$ the function  $\fou{1_K}$ is the product of a nontrivial polynomial and an entire function.}
\smallskip

Let us introduce some notation. Given a polynomial $P(\ze)=\sum_{|l|\leq m}c_l \ze^l$, where $m$ is a positive integer,  $l=(l_1,\dots,l_n)$ denotes a multi-index, $c_l\in\C$,  $|l|=l_i+\dots+l_n$ and $\ze^l=\ze_1^{l_1}\dots\ze_n^{l_n}$, let  $P(D)$ denote the differential operator
\[
P(D)=\sum_{|l|\leq m}(\ii)^{-|l|}c_l\left(\pa^{l_1}/\pa x_1^{l_1}\right)\dots\left(\pa^{l_n}/\pa x_n^{l_n}\right),
\]
where $\pa^{0}/\pa x_i^{0}$ denotes the identity operator.
\cite[Theorem~8.4]{Rudin-91} states that
\[
\fou{1_K}=Pf,
\]
with $f$ entire  and $P$ a polynomial, if and only if the problem
\begin{equation}\label{differential_problem}
 P(D)u=1_K,
\end{equation}
has a solution $u$ in the class of distributions with support contained in $K$. Here $\fou{u}=f$ and~\eqref{differential_problem} has to be understood in the sense of distributions. The Theorem of supports for convolutions~\cite[Theorem~4.3.3]{Hormander-1983} and elementary considerations imply that if a solution $u$ to~\eqref{differential_problem} exists then its support is $K$. 

A particular instance of this problem has received much attention. When $P(\ze)=\ze_1^2+\dots+\ze_n^2-c$, for some $c>0$, \eqref{differential_problem} becomes
\begin{equation}\label{schiffer_conj}
\begin{cases}
 \Delta u+c u=-1 &\text{in $K$}\\
 u=\frac{\pa u}{\pa \nu}=0 &\text{on $\pa K$}
\end{cases}
\end{equation}
($\nu$ denotes the exterior normal to $\pa K$). Let $E\subset\R^n$ be a bounded simply connected Lipschitz domain. The Pompeiu Problem is a conjecture asserting that there exists  a nonzero continuous function $f:\R^n\to\R$ such that
\[
\int_{\cT(E)} f\, dx=0\quad\text{for all rigid motions $\cT$ in $\R^n$}
\]
only when $E$ is a ball. It is known that the Pompeiu Problem is equivalent to proving that a solution to~\eqref{schiffer_conj} (with $K$ replaced by $E$) exists for some $c>0$ only if $E$ is a ball (see Berenstein~\cite{Berenstein-1980}). Up to our knowledge these problems are still open.

The example of a ball implies that the irreducibility condition is not necessary for determination by covariogram. Indeed, when $K$ is a ball a solution to~\eqref{schiffer_conj} exists and $\fou{1_K}$ factors. On the other hand, in any dimension a ball $K$ is uniquely determined by $g_K$, as  Theorem~\ref{teo_radial_symmetry} implies.

\bibliographystyle{amsplain}

\end{document}

%% file: quattro_paral_ver2.pspdftex
\begin{picture}(0,0)%
\includegraphics{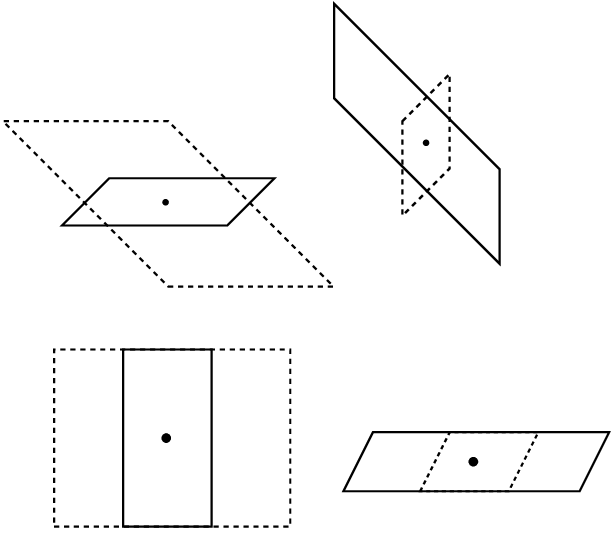}%
\end{picture}%
\setlength{\unitlength}{4144sp}%
\begingroup\makeatletter\ifx\SetFigFont\undefined%
\gdef\SetFigFont#1#2#3#4#5{%
  \reset@font\fontsize{#1}{#2pt}%
  \fontfamily{#3}\fontseries{#4}\fontshape{#5}%
  \selectfont}%
\fi\endgroup%
\begin{picture}(4664,4070)(-1266,-1700)
\put(856,929){\makebox(0,0)[lb]{\smash{{\SetFigFont{12}{14.4}{\familydefault}{\mddefault}{\updefault}$\mathcal H_1$}}}}
\put(1306,209){\makebox(0,0)[lb]{\smash{{\SetFigFont{12}{14.4}{\familydefault}{\mddefault}{\updefault}$\mathcal H_2$}}}}
\put(  1,-1096){\makebox(0,0)[lb]{\smash{{\SetFigFont{12}{14.4}{\familydefault}{\mddefault}{\updefault}$o$}}}}
\put(2341,-1276){\makebox(0,0)[lb]{\smash{{\SetFigFont{12}{14.4}{\familydefault}{\mddefault}{\updefault}$y'$}}}}
\put(  1,704){\makebox(0,0)[lb]{\smash{{\SetFigFont{12}{14.4}{\familydefault}{\mddefault}{\updefault}$o$}}}}
\put(1981,1154){\makebox(0,0)[lb]{\smash{{\SetFigFont{12}{14.4}{\familydefault}{\mddefault}{\updefault}$y$}}}}
\put(966,-1631){\makebox(0,0)[lb]{\smash{{\SetFigFont{12}{14.4}{\familydefault}{\mddefault}{\updefault}$\mathcal H_4$}}}}
\put(2557,389){\makebox(0,0)[lb]{\smash{{\SetFigFont{12}{14.4}{\familydefault}{\mddefault}{\updefault}$\mathcal K_1$}}}}
\put(2176,1689){\makebox(0,0)[lb]{\smash{{\SetFigFont{12}{14.4}{\familydefault}{\mddefault}{\updefault}$\mathcal K_2$}}}}
\put( 74,-481){\makebox(0,0)[lb]{\smash{{\SetFigFont{12}{14.4}{\familydefault}{\mddefault}{\updefault}$\mathcal H_3$}}}}
\put(2491,-1086){\makebox(0,0)[lb]{\smash{{\SetFigFont{12}{14.4}{\familydefault}{\mddefault}{\updefault}$\mathcal K_4$}}}}
\put(3236,-1361){\makebox(0,0)[lb]{\smash{{\SetFigFont{12}{14.4}{\familydefault}{\mddefault}{\updefault}$\mathcal K_3$}}}}
\end{picture}%